\documentclass[12pt]{amsart}
\usepackage{amsfonts}
\usepackage{amsmath}
\usepackage{amssymb}
\usepackage[margin=1in]{geometry}
\usepackage{tikz}
\usepackage{amsthm}   							% hiermee kan je zelf omgevingen zoals Stelling definieëren
\usepackage{mathtools}							% uitbreiding van amsmath 
\usepackage{commath}							% Om differentialen te typesetten als operatoren:
\usepackage[english]{babel}						% http://ftp.snt.utwente.nl/pub/software/tex/macros/latex/required/babel/base/babel.pdf
\usepackage[hidelinks]{hyperref}  					% voor hyperlinks in de pdf

\makeatletter
\@namedef{subjclassname@2020}{\textup{2020} Mathematics Subject Classification}
\makeatother

\newtheorem{theorem}{Theorem}[section]
\newtheorem{corollary}[theorem]{Corollary}
\newtheorem{lemma}[theorem]{Lemma}
\newtheorem{proposition}[theorem]{Proposition}
\newtheorem{conjecture}{Conjecture}

\theoremstyle{definition}

\newtheorem{remark}[theorem]{Remark}

\numberwithin{equation}{section}

\newcommand{\N}{\mathbb{N}}
\newcommand{\Z}{\mathbb{Z}}

\newcommand{\R}{\mathbb{R}}

\renewcommand{\Re}{\operatorname{Re}}
\renewcommand{\Im}{\operatorname{Im}}

\newcommand{\I}{\mathrm{i}}
\newcommand{\e}{\mathrm{e}}

\newcommand{\eps}{\varepsilon}
\newcommand{\vphi}{\varphi}

\newcommand{\MP}{\mathcal{P}}
\newcommand{\MN}{\mathcal{N}}
\newcommand{\LH}{\mathrm{LH}}

\DeclareMathOperator{\Li}{Li}
\DeclareMathOperator{\supp}{supp}

\begin{document}

\title{On the Lindel\"of hypothesis for general sequences}

\author[F. Broucke]{Frederik Broucke}
\thanks{F. Broucke was supported by a postdoctoral fellowship (grant number 12ZZH23N) of the Research Foundation -- Flanders}
\address{Department of Mathematics: Analysis, Logic and Discrete Mathematics\\ Ghent University\\ Krijgslaan 281\\ 9000 Gent\\ Belgium}
\email{fabrouck.broucke@ugent.be}

\author[S. Weish\"aupl]{Sebastian Weish\"aupl}
\thanks{S. Weish\"aupl was affiliated with the University of W\"urzburg during part of the research.}
\email{weishaeupl.sebastian@gmail.com}

% 11M26: Nonreal zeros of $\zeta(s)$ and $L(s, \chi)$; Riemann and other hypotheses.
% 11L03: Trigonometric and exponential sums, general.
% 11N80: Generalized primes and integers.

\begin{abstract} 
In a recent paper, Gonek, Graham, and Lee introduced a notion of the Lindel\"of hypothesis (LH) for general sequences which coincides with the usual Lindel\"of hypothesis for the Riemann zeta function in the case of the sequence of positive integers. They made two conjectures: that LH should hold for every admissible sequence of positive integers, and that LH should hold for the ``generic'' admissible sequence of positive real numbers. In this paper, we give counterexamples to the first conjecture, and show that the second conjecture can be either true or false, depending on the meaning of ``generic'': we construct probabilistic processes producing sequences satisfying LH with probability 1, and we construct Baire topological spaces of sequences for which the subspace of sequences satisfying LH is meagre. We also extend the main result of Gonek, Graham, and Lee, stating that the Riemann hypothesis is equivalent to LH for the sequence of prime numbers, to the context of Beurling generalized number systems.
\end{abstract}

\maketitle

\section{Introduction}
The \emph{Lindel\"of hypothesis} is a conjecture about the growth of the Riemann zeta function $\zeta(s)$. It says that for each $\eps>0$,
\[
	\zeta(1/2 + \I \tau) \ll_{\eps} \abs{\tau}^{\eps}, \quad \text{ as } \abs{\tau}\to\infty.
\]
Using Perron inversion and, conversely, the approximation of $\zeta(s)$ by partial sums of $\sum_{n} n^{-s}$, one can show (see e.g.\ the appendix of \cite{GGL} or the Appendix below) that the Lindel\"of hypothesis is equivalent to the following statement: for every $\eps>0$ and $B>0$,
\[
	\sum_{n\le x}n^{-\I t} = \frac{x^{1-\I t}}{1-\I t} + O_{\eps,B}(x^{1/2}\abs{t}^{\eps}),
\]
for $1\le x\le \abs{t}^{B}$.
Based on this formulation, Gonek, Graham, and Lee introduced in the article \cite{GGL} a notion of the Lindel\"of hypothesis for general ``admissible'' sequences $\MN$. They call an increasing sequence of positive reals $\MN = (n_{j})_{j\ge1}$ with counting function $N(x) = \sum_{n_{j}\le x}1$ \emph{admissible} if $N(x)$ can be written as $N(x) = M(x) + E(x)$, where
\begin{enumerate}
	\item $M(x)$ is differentiable, 
	\item $x^{\alpha_{1}} \ll M(x) \ll x^{\alpha_{2}}$ for some $0 < \alpha_{1} < \alpha_{2}$, and
	\item $E(x) \ll_{A} M(x)/(\log x)^{A}$ for every positive constant $A$.
\end{enumerate}
The Lindel\"of hypothesis is true for an admissible sequence $\MN$ with associated function $M(x)$, or $\LH(\MN, M(x))$ for short, if for every $\eps >0$ and $B>0$,
\[
	\sum_{n_{j}\le x}n_{j}^{-\I t} = \int_{n_{1}}^{x}u^{-\I t}M'(u)\dif u + O_{\eps, B}\bigl(M(x)^{1/2}\abs{t}^{\eps}\bigr),
\]
for $n_{1} \le x \le \abs{t}^{B}$. With this notation, the classical Lindel\"of hypothesis is equivalent to the statement $\LH(\N_{>0}, x)$.

Note that we employ the notation of Banks \cite{Banks} and explicitly mention the function $M(x)$, as the validity of $\LH$ also depends on this function. (For example, the statements $\LH(\N_{>0}, x)$ and $\LH(\N_{>0}, x+x^{3/4})$ are mutually exclusive.) %In fact, for a sequence $\MN$ which does not have too large clusters, one can always artificially construct a function $M(x)$ such that $\LH(\MN, M(x))$ holds. See Section \ref{sec: remarks definition} for a precise statement and proof.
 
 \medskip
 
%In \cite{GGL} it is shown that $\LH(\N_{>0}, x)$ is equivalent to the classical Lindel\"of hypothesis, i.e.\ the statement that
%\[
%	\zeta(1/2+\I t) \ll_{\eps} \abs{t}^{\eps}
%\]
%for every $\eps>0$ (see the appendix for a generalization to general sequences with $N(x) = Ax + O(x^{\theta})$, where $\theta < 1/2$). 
The main result of \cite{GGL} is
\begin{theorem}[Gonek, Graham, Lee]
\label{GGL theorem}
Let $\mathbb{P}$ denote the sequence of prime numbers. Then $\LH(\mathbb{P}, \Li(x))$ is equivalent to the Riemann hypothesis.
\end{theorem}
Here we used $\Li(x)$ to denote the logarithmic integral, which we define as $\Li(x) = \int_{1}^{x}\frac{1-u^{-1}}{\log u}\dif u$.
Interestingly, RH is also equivalent to LH for certain thin sets of primes. For example, in \cite{Banks} Banks showed that there are arbitrarily small values of $\eps>0$ and subsequences $\MP_{\eps}$ of the primes with counting function $\sim \eps\Li(x)$ such that the Riemann hypothesis is equivalent to $\LH(\MP_{\eps}, \eps\Li(x))$.
 
In \cite{GGL}, the following two conjectures are made.
\begin{conjecture}[Gonek, Graham, Lee]
\label{conjecture 1}
$\LH(\MN)$ is true for every admissible sequence of positive integers $\MN$.
\end{conjecture}
\begin{conjecture}[Gonek, Graham, Lee]
\label{conjecture 2}
$\LH(\MN)$ is true for the ``generic'' admissible sequence of positive real numbers $\MN$.
\end{conjecture}
The aim of this paper is twofold. First, we resolve the above conjectures. Note that the conjectures are not completely unambiguous, as no function $M(x)$ is mentioned in their statements, and a precise meaning of ``generic'' is not specified. However, we feel that the results we obtained provide compelling answers to what the authors of \cite{GGL} likely had in mind when formulating their conjectures.

Concretely, we provide the following counterexample to Conjecture \ref{conjecture 1}.
\begin{theorem}
\label{counterexample conjecture 1}
There exist integer sequences $\MN = (n_{j})_{j\ge1}$ given by $n_{j} = 3j + \delta_{j}$, where $\delta_{j}=0,1$ or $2$, with counting function $N(x) = x/3 + O(1)$, and for which $\LH(\MN, x/3)$ is false.
\end{theorem}
Actually, one can even obtain integer sequences $\MN$ of density $1$ failing $\LH(\MN,x)$, see Theorem \ref{th: counterexample density 1} below.

Regarding Conjecture \ref{conjecture 2}, we adopt two viewpoints commonly used when addressing questions concerning ``generic'' or ``typical'' behavior: a probabilistic (or measure-theoretical) viewpoint and a topological viewpoint.

For suitable functions $M(x)$, we construct natural probabilistic processes (natural measure spaces) which produce sequences $\MN$ satisfying $\LH(\MN, M(x))$ with probability $1$ (for which the subspace of sequences $\MN$ satisfying $\LH(\MN,M(x))$ has full measure); see Theorems \ref{probabilistic theorem} and \ref{probabilistic theorem squareroot}. On the other hand, we also construct natural topological spaces of sequences for which the subspace of sequences $\MN$ satisfying $\LH(\MN, M(x))$ is actually negligible in the sense of Baire category; see Theorem \ref{categoric theorem}. Both viewpoints can also be combined and applied to a single locally compact Hausdorff space equipped with a (Radon) probability measure.
Hence we demonstrate that Conjecture \ref{conjecture 2} can have different answers depending on the specific meaning of the word ``generic''.

\medskip
The second aim is to extend Theorem \ref{GGL theorem} to Beurling generalized number systems. Here we just state the main theorem, and refer to the beginning of Section \ref{sec: Beurling} for the necessary background and notation regarding such number systems.

Given a Beurling number system $(\MP,\MN)$, we denote its integer-counting function by $N_{\MP}(x)$, its prime-counting function by $\pi_{\MP}(x)$ and its Chebyshev function by $\psi_{\MP}(x)$. We say that the Riemann hypothesis holds for the system $(\MP, \MN)$ if
\[
	\psi_{\MP}(x) = x + O_{\eps}(x^{1/2+\eps}), \quad \text{for every } \eps>0.
\]
\begin{theorem}
\label{th: LH(P) <=> RH}
Let $(\MP, \MN)$ be a system of Beurling generalized primes and integers, for which $N_{\MP}(x) = Ax + O(x^{1/2})$ for some $A>0$. Then $\LH(\MP, \Li(x))$ is equivalent to the Riemann hypothesis for $(\MP,\MN)$. 
\end{theorem}
We note that the exponent $1/2$ in the error term for $N_{\MP}$ is sharp, see Proposition \ref{counterexample weakening square root error term} below.

\medskip

The paper is organized as follows. In Section \ref{sec: LH general sequences}, we answer the conjectures of Gonek, Graham, and Lee. Next, in Section \ref{sec: Beurling}, we prove Theorem \ref{th: LH(P) <=> RH}, comment on the necessity of the condition $N_{\MP}(x) = Ax+O(x^{1/2})$, and show the existence of very ``nice'' Beurling systems, in the sense that the Riemann hypothesis, $\LH(\MN, Ax)$, and $\LH(\MP, \Li(x))$ all hold. In Section \ref{sec: remarks definition} we make some remarks on the definition of $\LH(\MN, M(x))$. In particular, we show that under mild conditions on the sequence $\MN$, one can always find a smooth function $M(x)$ such that $\LH(\MN, M(x))$ holds (see Proposition \ref{prop: choice M}). Finally, in the appendix, we sketch the equivalence of $\LH(\MN, Ax)$ and ``$|\tau|^{\eps}$-bounds'' for the associated Dirichlet series of $\MN$, for a natural class of sequences $\MN$.

\medskip

Regarding notation, we will typically write complex numbers as $s=\sigma+\I \tau$ and reserve the letter $t$ for the parameter occurring in the definition of $\LH$. When considering a sequence $\MN$, we will always use $N(x)$ to denote its counting function. Sometimes we will also speak about $\LH(\MN,M(x))$ for sequences $\MN$ which are not necessarily non-decreasing.
%%%%%%%%%%%%%%%%%%%%%%%%%%%%%%%%%%%%%%%%%%%%%%%%%%%%%%%%%%%%%%%%%%%%%
\section{The Lindel\"of hypothesis for general sequences}
\label{sec: LH general sequences}

\subsection{Integer sequences failing LH}
\label{subsec: Integer sequences}
In \cite[p. 2866]{GGL} an example due to Montgomery is given of a sequence $\MN$ satisfying $N(x) = x+O(1)$ and for which $\LH(\MN, x)$ is false. The idea is to consider sequences $(t_{k})_{k \ge 1}$ and $(x_{k})_{k \ge 1}$ and ``cook up'' the elements $n_{j}$ of $\MN$ with $x_{k}/2<n_{j}\le x_{k}$ in such a way that all the $n_{j}^{-\I t_{k}}$ point in the same direction, making the exponential sum $\sum_{n_{j}\le x_{k}}n_{j}^{-\I t_{k}}$ large (in fact even $\gg x_{k}$). 

A similar method can be employed to construct \emph{integer} sequences for which $\LH$ fails.

\begin{proof}[Proof of Theorem \ref{counterexample conjecture 1}]
We inductively construct an integer sequence $\MN = (n_j)_{j\ge1}$, where $n_j = 3j + \delta_j$ with $\delta_j \in \{0, 1, 2\}$. Let $ 0 < \alpha < \pi/6$, $\beta = (1 + \alpha/4\pi)^{-1} $ and suppose that $n_j$ has been selected for $j \le \beta K$, where $K$ is a large integer. %(the first members of the sequence can be selected at random)
We will now choose $n_j$ for $ \beta K < j \le K $. Set 
\[
	S_{J, t} = \sum_{j \leq J} n_j^{-\I t}.
\]
Let $t = 2 \pi K$ and suppose first that $\Re S_{\beta K, t} \ge 0$. We have
\begin{equation}\label{eq: counterexample 1}
	n_j^{-\I t} = (3j + \delta_j)^{-\I t} = (3j)^{-\I t} \exp\Big\{-\I t \log\Big(1 + \frac{\delta_j}{3j}\Big)\Big\}.
\end{equation}
Here
\[
	t \log\Big(1 + \frac{\delta_j}{3j}\Big) = \frac{t \delta_j}{3j} + O\Big(\frac{t}{j^2}\Big) = \frac{t \delta_j}{3j} + O\Big(\frac{1}{K}\Big), \quad \text{if } j\ge \beta K.
\]
Hence for sufficiently large $K$ it follows that
\begin{equation}\label{eq: counterexample 2}
	\abs{t \log\Big(1 + \frac{\delta_j}{3j}\Big) - \frac{t\delta_j}{3j}} \leq \frac{\alpha}{2}, \quad \text{for } j\ge \beta K.
\end{equation}
If now $\beta K < j \le K$ we have 
\begin{equation}\label{eq: counterexample 3}
	 \frac{2\pi \delta_j}{3} \le \frac{t \delta_j}{3j} = \frac{2 \pi K \delta_j}{3j} \le \beta^{-1} \frac{2\pi \delta_j}{3} = \Big(1 + \frac{\alpha}{4\pi}\Big) \frac{2\pi \delta_j}{3} \le \frac{2\pi \delta_j}{3} + \frac{\alpha}{2}
\end{equation}
Combining (\ref{eq: counterexample 2}) and (\ref{eq: counterexample 3}) yields
\begin{equation}
\label{eq: counterexample 4}
	\abs{t \log\Big(1 + \frac{\delta_j}{3j}\Big) - \frac{2\pi \delta_j}{3}} \leq \alpha, \quad \text{for } \beta K < j \le K.
\end{equation}
Now it's geometrically easy to see that in view of \eqref{eq: counterexample 1} and \eqref{eq: counterexample 4} we can choose $\delta_j \in \{0, 1, 2\}$ for all $\beta K < j \leq K$ in such a way that $\Re n_j^{-\I t} > c \coloneqq \cos(\alpha + \pi/3) > 0$. Indeed, for three equidistant arcs of length $2\alpha$ on the unit circle, at least one arc lies completely within the region $\bigl\{(x,y) \in \R^{2}: x \ge \cos(\alpha+\pi/3)\bigr\}$.
%The proof relies on the following geometric fact, which is easy to verify. Let $0 < \alpha < \pi/6$. Then for arbitary $x_0, x_1, x_2 \in \R$ with $\abs{x_j - 2 \pi j / 3} \leq \alpha$ and all $t \in \R$ one of the three points $\exp(\I (t + x_j))$, $j=0, 1, 2$ will have real part at least $c := \cos(\alpha + \pi/3) > 0$ and one of them will have real part at most $-c$.
It follows that
\[
	\Re S_{K, t} \ge (K - \beta K - 1) c + \Re S_{\beta K, t} \ge (1 - \beta) c K - c,
\]
hence $S_{K, t} \gg K$. In the case $\Re S_{\beta K}(t) < 0$ we choose $\delta_j$ for $\beta K < j \le K$ such that $\Re(n_j^{-\I t}) \le - c < 0$, so $S_{K, t} \gg K$ follows analogously. In this way we can select the sequence $(\delta_{j})_{j\ge1}$ inductively such that for infinitely many integers $K$ we have that $S_{K, t} \gg K$, where $t = 2 \pi K$. %\vlt{In this way we can construct sequences $(J_{k})_{k\ge 1}$ (e.g.\ $J_{k} = J_{0}/\beta^{k}$), $(t_{k})_{k\ge1}=(2\pi J_{k})$, and $(\delta_{j})_{j\ge1}$ for which $S_{\lfloor J_{k}\rfloor}(t_{k}) \gg J_{k}$.}
Hence $\LH(\MN, x/3)$ is false.
\end{proof}
%
%\rge{
%[
%The above construction can be generalized to create integer sequences $\MN$ with other densities $\lambda$ such that $\LH(\MN, \lambda x)$ fails. Observe that for any positive integer $k$ and $0 < \alpha < \frac{\pi}{2k(2k+1)}$, among $(2k+1)$ equidistant arcs of length $2\alpha$ on the unit circle, at least $k$ are contained in the region $\bigl\{(x,y)\in\R^{2}: x\ge \sin\bigl(\frac{\pi}{4k+2}-k\alpha\bigr)\bigr\}$. This allows us to construct integer sequences $\MN$ with $N(x) = \frac{k}{2k+1}x+O(1)$ such that the exponential sums can be $\gg x$: for each $j$, we take the $k$ distinct integers of the form $(2k+1)j + \delta$, $\delta \in \{0,1,\dotsc, 2k\}$ corresponding to those $k$ arcs.]}
%
%\rge{[ We can even further optimize this argument. Let $m$ be a large integer and $k = m - l$, where $l=o(m)$. Then the sum of the $k$ $m$-th roots of unity which are the most pointing to the right is at least $0- (-l) = l$. So we can select the $k$ distinct integers of the form $mj+ \delta$, $\delta\in \{0,1,\dotsc, m-1\}$ in this way, so that we obtain an integer sequence $\MN$ with $N(x) = (k/m)x + O(1)$ and again having exponential sums as large as $(l/m)x$ for certain values of $x$ and $t$. Letting $m, k$ tend to infinity in one single construction, it might even be possible to show something like:]
%}

The proof of the preceding theorem can be generalized as follows. Instead of considering intervals of length $3$ from which we select $1$ integer, we can consider intervals of lenght $m\ge4$ from which we select $k<m$ integers. As 
\[
	\Re \sum_{\abs{j} \le k/2}\e^{\I j \frac{2\pi}{m}} \gg \min(k, m-k),
\] 
we can construct the sequence $\MN = \MN_{m,k}$ in such a way that 
\begin{align*}
	N(x) 										&= \frac{k}{m}x + O(1), 						&&\quad \text{for all } x \ge 1,\quad \text{and}\\
	\abs[3]{\sum_{\substack{n\le x\\n\in \MN}}n^{-\I t}} 	&\gg \min\Bigl(\frac{k}{m}, \frac{m-k}{m}\Bigr)x, 		&&\quad\text{on certain (infinite) subsequences of $x$ and }t\sim\frac{2\pi x}{m}.  
\end{align*}
If we let $k\sim m\to \infty$ in a single construction, we can even achieve integer sequences of density $1$ failing $\LH$:
\begin{theorem}
\label{th: counterexample density 1}
For each $\eps>0$, there exists an integer sequence $\MN$ satisfying $N(x) = x + O(x^{1/2+\eps})$ for which $\LH(\MN, x)$ fails.
\end{theorem}

\begin{proof}
Let $\eps>0$ be a fixed small number. We inductively define $\MN=\MN_{\eps}$ by removing certain well-chosen integers from $\N_{>0}$. Let $M$ be a large integer and set 
\[
	m\coloneqq \lfloor M^{1/2-\eps}\rfloor, \quad l\coloneqq \lfloor m^{1-\eps}\rfloor, \quad k\coloneqq m-l, \quad \text{and} \quad 1/\beta \coloneqq 1 + \frac{1}{8m}. 
\]
We divide the interval $[M, M/\beta)$ into subintervals $I_{j} = \bigl[M+(j-1)m, M+jm\bigr)$ for $j=1, 2, \dotsc, J$, where $J=\lfloor (1/\beta-1)M/m\rfloor \asymp M^{2\eps}$. From each subinterval $I_{j}$ we will delete about $l$ integers (from the final interval $\bigl[M+Jm, M/\beta\bigr)$ we will delete no integers).

Let $t\coloneqq \frac{2\pi M}{\beta m} \asymp M^{1/2+\eps}$, and suppose without loss of generality that $\Re \sum_{n=1}^{M-1}n^{-\I t} > 0$. Each integer $n\in I_{j}$ is of the form $n = M + (j-1)m + \delta$ for some $\delta\in \{0,1,\dotsc, m-1\}$. We have
\[
	n^{-\I t} = \bigl(M+(j-1)m\bigr)^{-\I t}\exp\Bigl\{-\I t\log \Bigl(1+\frac{\delta}{M+(j-1)m}\Bigr)\Bigr\},
\]
where 
\[
	t\log \Bigl(1+\frac{\delta}{M+(j-1)m}\Bigr) = \frac{t\delta}{M+(j-1)m} + O\Bigl(\frac{1}{mM^{2\eps}}\Bigr).
\]
Also 
\[
	\frac{2\pi\delta}{m} \le \frac{t\delta}{M+(j-1)m} \le \frac{2\pi\delta}{\beta m} = \frac{2\pi\delta}{m} + \frac{2\pi\delta}{8m^{2}},
\]
so that 
\begin{equation}
\label{eq: phase error}
	\abs{t\log \Bigl(1+\frac{\delta}{M+(j-1)m}\Bigr) - \frac{2\pi\delta}{m} } \le \frac{2\pi}{4m},
\end{equation}
provided that $M$ is sufficiently large (depending on $\eps$).

For each $j$, we now choose $\delta_{j}$ so that $\Re \bigl(M + (j-1)m+\delta_{j}\bigr)^{-\I t}$ is maximal; in fact this maximal value of the real part is $1+O(1/m)$. From the interval $I_{j}$, we keep the integers $n$ corresponding to 
\[
	\delta = \delta_{j}, \quad \delta = \delta_{j} \pm 1, \quad \dotsc, \quad \delta = \delta_{j} \pm \lfloor k/2\rfloor \mod m, \quad \text{say } \delta\in \Delta_{j}.
\]
 This yields $k$ or $k+1$ integers, depending on whether $k$ is odd or even, and we delete the other $l$ or $l-1$ integers from $I_{j}$. We have
\[
	\sum_{\delta \in \Delta_{j}}\bigl(M+(j-1)m+\delta\bigr)^{-\I t} 
	= \bigl(M+(j-1)m+\delta_{j}\bigr)^{\I t}\sum_{\delta\in\Delta_{j}} \Bigl(\frac{M+(j-1)m+\delta}{M+(j-1)m+\delta_{j}}\Bigr)^{-\I t}.
\]
Now 
\begin{align*}
	\Re \sum_{\delta\in \Delta_{j}} \Bigl(\frac{M+(j-1)m+\delta}{M+(j-1)m+\delta_{j}}\Bigr)^{-\I t} 
		&\ge 1 + 2\cos\Bigl(\tfrac{3}{2}\cdot\tfrac{2\pi}{m}\Bigr)+\dotsb + 2\cos\Bigl(\bigl(\tfrac{1}{2}+\lfloor \tfrac{k}{2}\rfloor\bigr)\cdot \tfrac{2\pi}{m}\Bigr)\\
		&= l + O\bigl(m^{1-2\eps}\bigr),
\end{align*}
where the first line follows from \eqref{eq: phase error}, and the second line from a short calculation. Hence we find 
\[
	\Re \sum_{\delta\in\Delta_{j}}\bigl(M+(j-1)m+\delta\bigr)^{-\I t} \ge l\bigl(1+O(m^{-\eps})\bigr).
\]
From this it follows that 
\[
	\abs[3]{\sum_{\substack{n\le M/\beta\\ n\in \MN}}n^{-\I t}} \ge Jl\bigl(1+O(m^{-\eps})\bigr) + O(m) \gg M^{1/2+\eps/2 + \eps^{2}}.
\]
Also, for $M\le x< M/\beta$ with $x\in I_{j}$
\begin{align*}
	N(x) 	&= \sum_{\substack{n\le x\\n\in \MN}}1 = N(M) + (j-1)k+O(J+m) = N(M) + (j-1)m + O(Jl) \\
		&= N(M) + (x-M) +O\bigl(M^{1/2+\eps/2 + \eps^{2}}\bigr) = N(M) + (x-M) + O\bigl(x^{1/2+\eps/2 + \eps^{2}}\bigr).
\end{align*}

Inductively applying this deleting procedure on intervals $[M_{\nu}, M_{\nu}/\beta_{\nu})$ for a sequence of $(M_{\nu})_{\nu\ge1}$ with $M_{\nu+1}>M_{\nu}/\beta_{\nu}$ and corresponding sequence $(t_{\nu})_{\nu\ge1}$, $t_{\nu} \asymp M_{\nu}^{1/2+\eps}$, yields a sequence of integers $\MN$ satisfying
\[
	N(x) = x + O\bigl(x^{1/2+\eps/2 + \eps^{2}}\bigr), \quad \sum_{\substack{n\le M_{\nu}/\beta_{\nu}\\n\in \MN}}n^{-\I t_{\nu}} \gg \bigl(M_{\nu}/\beta_{\nu}\bigr)^{1/2+\eps/2 + \eps^{2}}.
\]
In particular, $\LH(\MN, x)$ fails already for $B=2$.
\end{proof}

\begin{remark}
For an increasing integer sequence $\MN$ satisfying $N(x) = x + O_{\epsilon}(x^{1/2+\epsilon})$ for every $\epsilon>0$, we have 
\[
	\sum_{\substack{n\le x\\n\in\MN}}n^{-\I t} = \sum_{\substack{n\le x\\n\in \N_{>0}}}n^{-\I t} + O_{\eps, B}(x^{1/2}\abs{t}^{\eps}),
\]
if $1\le x\le \abs{t}^{B}$ (take $\epsilon = \eps/B$),
so that $\LH(\MN, x)$ is equivalent to $\LH(\N_{>0}, x)$.
\end{remark}

One could argue that introducing this ``conspiracy'' among the $n_{j}$ is artificial and not typical behavior, and that the ``generic'' sequence should still satisfy $\LH$, as per Conjecture \ref{conjecture 2}. To make the notion of generic explicit, we consider two viewpoints: a probabilistic viewpoint and a topological viewpoint. The results of the former viewpoint can be interpreted as a positive answer to Conjecture \ref{conjecture 2}, while the results of the latter viewpoint can be interpreted as a negative answer.

\subsection{Probabilistic viewpoint} 

Let $x_{0}>0$ and $M(x)$ be a right-continuous non-decreasing unbounded function supported on $[x_{0},\infty)$ which satisfies $M(x_{0})=0$ and $M(x)\ll x^{\alpha_{2}}$ for some $\alpha_{2}>0$. 
%Suppose also that 
%\[
%	\int_{x_{0}}^{x}\frac{\sqrt{M(u)}}{u}\dif u \ll \sqrt{M(x)}.
%\]
%Note that this implies in particular that $M(x) \gg x^{\delta}$ for some $\delta>0$. 
A natural probabilistic process to generate sequences with counting function $N(x)$ close to $M(x)$ is as follows.
Set $x_{j} = \inf\{x: M(x) \ge j\}$ for $j= 1, 2, \dotsc$. Assume that\footnote{We only assume this for simplicity of the argument: if $M(x_{j}-) < j < M(x_{j}+)$ and $k\ge0$ is such that $x_{j} = x_{j+1} = \dotso = x_{j+k} < x_{j+k+1}$, then one can distribute the ``mass'' of $\dif M(x)$ over the sets $(x_{j-1}, x_{j}]$, $\{x_{j}\}, \dotsc, \{x_{j}\}$ ($k$ times), and $(x_{j}, x_{j+1}]$ in an obvious way and a similar result holds (see \cite[Theorem 1.2]{BrouckeVindas}).} $M(x_{j}) = j$ for every $j$. Then $\dif M(x)$ defines a probability measure on the interval $(x_{j-1}, x_{j}]$ for each $j\ge1$.

\begin{theorem}
\label{probabilistic theorem}
With $M(x)$ and $x_{j}$ as above, select $n_{j}$ randomly and independently from the interval $(x_{j-1}, x_{j}]$ according to the probability distribution $\dif M(x)\big|_{(x_{j-1}, x_{j}]}$. Then with probability $1$ we have for each $x\ge x_{0}$ and $t\in \R$
\begin{equation}
\label{eq: probabilistic bound}
	\abs[3]{\sum_{n_{j}\le x}n_{j}^{-\I t} - \int_{x_{0}}^{x}u^{-\I t}\dif M(u) } \ll \sqrt{M(x)}\bigl(\sqrt{\log(x+1)} + \sqrt{\log(\,\abs{t}+1)}\bigr)\log(x+1).
\end{equation}
\end{theorem}
Note that the sequences $\MN=(n_{j})_{j\ge1}$ generated by the probabilistic process all satisfy $\abs{N(x) - M(x)} \le 1$, so if $M$ is differentiable and satisfies additionally $M(x) \gg x^{\alpha_{1}}$ for some $\alpha_{1}>0$, all these sequences are admissible and $\LH(\MN, M(x))$ holds with probability $1$. 

This theorem is a slight modification of \cite[Theorem 1.2]{BrouckeVindas}, which corresponds to the case $M(x) \ll x/\log x$ (see also Remark 2.3 in that paper). The proof is nearly identical to the proof given in \cite{BrouckeVindas}, but we provide it here also for convenience of the reader.

\begin{proof}
Let $(N_{j})_{j\ge1}$ be a sequence of independent random variables, where $N_{j}$ is distributed according to the probability measure $\dif M\big|_{(x_{j-1},x_{j}]}$. For real $t$ we define the random variable $X_{j,t} = N_{j}^{-\I t}$, with expectation
\[
	E(X_{j,t}) = \int_{x_{j-1}}^{x_{j}} u^{-\I t} \dif M(u).
\]
The sum $S_{J, t} = \sum_{j=1}^{J}X_{j,t}$ has expected value $E(S_{J, t}) = \int_{x_{0}}^{x_{J}} u^{-\I t} \dif M(u)$, and we can bound the probability that $S_{J, t}$ is far away from this value using an inequality of Hoeffding \cite{Hoeffding}. This inequality states that for independent real-valued random variables $Y_{j}$ with $a_{j}\le Y_{j}\le b_{j}$ and with sum $S = \sum_{j=1}^{J}Y_{j}$, we have for every positive $v$ that
\[
	P\bigl(S-E(S) \ge v\bigr) \le \exp\Biggl(-\frac{2v^{2}}{\sum_{j=1}^{J}(b_{j}-a_{j})^{2}}\Biggr).
\]
In our case, applying Hoeffding's inequality with $v_{J, t} = C\sqrt{J}\bigl(\sqrt{\log(x_{J}+1)}+\sqrt{\log(\,\abs{t}+1)}\bigr)$ to the random variables $\Re(X_{j, t}) = \cos(- t \log N_j)$ gives
\[
	P\Bigl(\Re\big(S_{J, t} - E(S_{J, t})\big) \ge v_{J, k}\Bigr) \le \exp\Bigl\{-\frac{C^{2}}{2}\bigl(\log(x_{J}+1)+\log(\,\abs{t}+1)\bigr)\Bigr\}.
\]
Applying the same principle to the random variables $-\Re(X_{j,t})$, $\pm \Im(X_{j,t})$, we obtain
\[
	P\Bigl(\, \abs[2]{S_{J, t} - E(S_{J, t})} \ge \sqrt{2}v_{J, k}\Bigr) \le \frac{4}{(x_{J}+1)^{C^{2}/2}(\,\abs{t}+1)^{C^{2}/2}},
\]
Let now $A_{J,k}$ denote the event $\abs{S_{J, k} - E(S_{J, k})} \ge \sqrt{2}v_{J, k}$. By assumption $M(x) \ll x^{\alpha_{2}}$, so that
\[
	\sum_{J=1}^{\infty}x_{J}^{-C^{2}/2} = \int_{x_{0}}^{\infty}u^{-C^{2}/2}\dif\, \lfloor M(u)\rfloor = \frac{C^{2}}{2}\int_{x_{0}}^{\infty}\lfloor M(u)\rfloor u^{-C^{2}/2-1}\dif u < \infty,
\]
provided that $C^{2}/2 > \alpha_{2}$. Hence, fixing some $C > \max\{\sqrt{2\alpha_{2}},\sqrt{2}\}$ we get that 
\[
	\sum_{J=1}^{\infty}\sum_{k=1}^{\infty}P(A_{J,k}) \le \sum_{J=1}^{\infty}\sum_{k=1}^{\infty} \frac{4}{(x_{J}+1)^{C^{2}/2}(k+1)^{C^{2}/2}} < \infty.
\]
By the Borel--Cantelli lemma, one has that with probability 1 only finitely many of the events $A_{J,k}$ occur. 

To complete the proof we show that a sequence $\MN = (n_{j})_{j\ge1}$ contained in only finitely many sets $A_{J,k}$ satisfies \eqref{eq: probabilistic bound}. Clearly this holds if $x=x_{J}$ and $t=k\in \Z$. Set $S(x,t) = \sum_{n_{j}\le x}n_{j}^{-\I t}$. Then we have for $x\in (x_{J-1},x_{J})$ 
\begin{equation}
\label{eq: critical eq}
	S(x,k) = S_{J-1, k} + O(1) = \int_{x_{0}}^{x}u^{-\I k}\dif M(u) + O\Bigl(\sqrt{M(x)}\bigl(\sqrt{\log(x+1)}+\sqrt{\log(\,\abs{k}+1)}\bigr)\Bigr).
\end{equation}
Finally, if $t \in (k-1,k)$, $k\in\Z$, then 
\begin{align*}
	S(x,t)	&= \int_{x_{0}}^{x^{+}}u^{-\I(t-k)}\dif S(u,k) = S(x,k)x^{-\I(t-k)} +\I(t-k)\int_{x_{0}}^{x}S(u,k)u^{-\I(t-k)-1}\dif u \\
			&= \int_{x_{0}}^{x}u^{-\I k}\dif M(u)x^{-\I(t-k)} + \I(t-k)\int_{x_{0}}^{x}\biggl(\int_{x_{0}}^{u}y^{-\I k}\dif M(y)\biggr)u^{-\I(t-k)-1}\dif u\\
			&\quad + O\Biggl(\sqrt{M(x)}\bigl(\sqrt{\log(x+1)} + \sqrt{\log(\,\abs{t}+1)}\bigr)\Biggr)\\
			 &\quad + O\Biggl(\int_{x_{0}}^{x}\frac{\sqrt{M(u)}}{u}\bigl(\sqrt{\log(u+1)} + \sqrt{\log(\,\abs{t}+1)}\bigr)\dif u\Biggr) \\
			&= \int_{x_{0}}^{x}u^{-\I t}\dif M(u) + O\Bigl(\sqrt{M(x)}\bigl(\sqrt{\log(x+1)} + \sqrt{\log(\,\abs{t}+1)}\bigr) \log(x+1) \Bigr),
\end{align*}
%by the assumption that $\int_{x_{0}}^{x}\sqrt{M(u)}/u\dif u \ll \sqrt{M(x)}$. 
This completes the proof.
\end{proof}

\begin{remark}
The last $\log$-factor in the estimate \eqref{eq: probabilistic bound} can be removed if one assumes $\int_{x_{0}}^{x}\sqrt{M(u)}\dif u/u \ll \sqrt{M(x)}$ (which incidentally also implies $M(x)\gg x^{\alpha_{1}}$ for some $\alpha_{1}>0$). This assumption is fulfilled for many commonly used functions.
\end{remark}

As an example, let us set $x_{0} = 1$ and $M(x) = x-1$. Then the above theorem implies in particular that choosing $n_{j}$, $j\ge1$, uniformly and independently from $(j,j+1]$ yields a sequence $\MN$ satisfying $\LH(\MN, x)$ with probability 1.

For another example related to the integer sequences we have constructed in Subsection \ref{subsec: Integer sequences}, let $k$ be a positive integer and apply the above theorem with $M(x) = (1/k)\lfloor x - k + 1\rfloor$ for $x \ge k$ and $M(x) = 0$ else. We obtain the following interesting dichotomy:
\begin{corollary}
Consider random sequences $\MN = (n_{j})_{j\ge1}$ given by $n_{j} = m j +\delta_{j}$, where the $\delta_{j}$ are selected uniformly and independently from the set $\{0,1,\dotsc,m-1\}$. %(i.e.\ $P(\delta_{j} = 1)=\dotso = P(\delta_{j}=k)=1/k$) 
If $\LH(\N_{>0},x)$ is true, then $\LH(\MN, x/m)$ holds with probability $1$. If on the other hand $\LH(\N_{>0}, x)$ is false, then $\LH(\MN, x/m)$ holds with probability $0$.
\end{corollary}

%In view of Threorem \ref{probabilistic theorem} one can ask for less restrictive ways of randomly selecting a sequence $\MN = (n_j)_j$. 

The ideas of Theorem \ref{probabilistic theorem} can be applied in more general settings. For example, consider $M(x)$ and $x_j$ as above and let $(j_{l})_{l\ge0}$ be an increasing sequence of integers with $j_{0}=0$. Instead of selecting $n_{j}$ randomly from $(x_{j-1},x_{j}]$, we select the $j_{l} - j_{l-1}$ points
\[
	n_{j_{l-1}+1},\; n_{j_{l-1} + 2},\dotsc ,n_{j_{l}} \quad \text{randomly and independently in the interval } (x_{j_{l-1}}, x_{j_{l}}],
\]
according to the probability distribution $\frac{1}{j_{l}-j_{l-1}}\dif M(x)\big|_{(x_{j_{l-1}}, x_{j_{l}}]}$.  If $j_{l}-j_{l-1} \ll \sqrt{j_{l-1}}$, 
%(which is fulfilled in particular for $j_l = \lfloor l^\theta \rfloor$, where $ 1 \leq \theta \leq 2$)
then with probability $1$ \eqref{eq: probabilistic bound} holds for each $x \geq x_{0}$ and $t \in \R$.
%\[
%	\abs[3]{\sum_{n_{j}\le x}n_{j}^{-\I t} - \int_{x_{0}}^{x}u^{-\I t}\dif M(u) } \ll \sqrt{M(x)}\bigl(\sqrt{\log(x+1)} + \sqrt{\log(\,\abs{t}+1)}\bigr) + M(x)^{1-\frac{1}{m}} \log(x).
%\]
Indeed, the proof can be adapted by restricting to the events $A_{j_{l},k}$, $l\ge1$, $k\in \Z$. The critical equation is \eqref{eq: critical eq}, where the $O(1)$ term now needs to be replaced by $O(j_{l}-j_{l-1})$ if $x \in (x_{j_{l-1}}, x_{j_{l}})$.

Any such random sequence described above has a counting function satisfying $N(x) = M(x) + O(E(x))$, where $E(x)$ is a non-decreasing function with $E(x_{j_{l-1}}) \ll j_{l}-j_{l-1}$. Hence if $M$ is differentiable and satisfies $M(x) \gg x^{\alpha_{1}}$ for some $\alpha_{1}>0$, all these sequences are again admissible and $\LH(\MN, M(x))$ holds with probability 1.
%\bl{More precisely it's
%\[
%	... \ll \sqrt{M(x)}\bigl(\sqrt{\log(x+1)} + \sqrt{\log(\,\abs{t}+1)}\bigr) + M(x)^{1-\frac{1}{m}} + \int_{x_0}^x \frac{M(u)^{1-\frac{1}{m}}}{u} \dif u.
%\]}
%The counting function $N(x)$ of $\MN$ fulfills $N(x) = M(x) + O(M(x)^{1 - 1/m})$. Hence, if $M$ is differentiable, the sequence $\MN$ is admissible and in the case $m = 2$ the statement $\LH(\MN, M(x))$ holds with probability $1$.

Another probabilistic result in favor of Conjecture \ref{conjecture 2} is the following.

\begin{theorem}
\label{probabilistic theorem squareroot}
	Let $A, K > 0$ and $0 < \theta < 1$. Select $n_{j}$ randomly and uniformly from the interval $\bigl(j/A - K j^\theta, j/A + K j^\theta\bigr] \cap \R_{>0}$ for every $j \ge 1$. Then the counting function $N(x)$ of the sequence $\MN = (n_j)_{j\ge1}$ satisfies $N(x) = A x + O(x^\theta)$ and we have with probability $1$ for each $x \geq 1$ and $ t \in \R$
	\begin{equation}\label{eq: probabilistic theorem squareroot}
		\sum_{n_{j}\le x}n_{j}^{-\I t} = A \frac{x^{1-\I t}}{1 - \I t} + O\Bigl(\sqrt{x}\bigl(\sqrt{\log(x+1)} + \sqrt{\log(\,\abs{t}+1)}\bigr)\Bigr) + O(x^{\theta} + x^{1 - \theta}).
	\end{equation}
	In particular in the case $\theta = 1/2$ the statement $\LH(\MN, Ax)$ holds with probability $1$. 
	%(after rearranging $\MN$ according to size).
\end{theorem}

%\bl{For the boundaries of the intervals also $j^{1/2} \log(j)$ and alike would work. With probability $1$ the sequence fulfills $N(x) = Ax + \Omega_{\pm}(x^\theta)$.}

\begin{proof}
By rescaling we can assume that $A=1$. Let $(N_j)_{j\ge1}$ be a sequence of independent random variables, where $N_j$ is uniformly distributed in the interval $I_j := (j - K j^\theta, j + K j^\theta] \cap \R_{>0}$. We define the random variable $X_{j, t} = N_j^{-\I t}$ for real $t$ with expectation
\[
	E(X_{j, t}) = \frac{1}{|I_{j}|} \int_{I_j} u^{-\I t} \dif u = \frac{1}{2Kj^{\theta}}\int_{I_{j}}u^{-\I t}\dif u,
\]
the last equality holding when $j \ge K^{1/(1-\theta)}$. The expectation of the sum $S_{J, t} = \sum_{j=1}^J X_{j,t}$ is therefore
\[
	E(S_{J, t}) %= \frac{1}{2K} \sum_{j=1}^J \frac{1}{j^\theta} \int_{I_j} u^{-it} \dif u 
	= \int_{1}^J f(u) u^{-it} \dif u + O(J^\theta), \quad \text{where } f(u) = \frac{1}{2K} \sum_{u \in I_j} j^{-\theta}.
\]
Now we want to see that $f(x)$ is close to $1$ and proceed as in the proof of Theorem \ref{probabilistic theorem}. For $x > 0$ let $L, M$ be the positive integers such that $x$ lies exclusively in the intervals $I_{L+1}, I_{L+2}, ..., I_M $. Then $x = L + O(L^\theta) = L + O(x^\theta)$, so
\[
	x = L + K L^\theta + O(1) = L + K \big(x + O(x^\theta)\big)^\theta + O(1) = L + K x^\theta + O(1 + x^{2\theta - 1}).
\]
Hence $L = x - K x^\theta + O(1 + x^{2\theta - 1})$ and similarly $M = x + K x^\theta + O(1 + x^{2\theta - 1})$. Now
\begin{align*}
	f(x) &= \frac{1}{2K} \sum_{j = L+1}^M j^{-\theta} = \frac{1}{2K} \sum_{x - K x^\theta < j \leq x + K x^\theta} j^{-\theta} + O(x^{-\theta} + x^{\theta -1})\\
	%&= \frac{1}{2K} \int_{x - K x^\theta}^{x + K x^\theta} u^{-\theta} \dif u  - \frac{1}{2K} \int_{x - K x^\theta}^{x + K x^\theta} u^{-\theta} \dif \, \{u\} + O(x^{-\theta} + x^{\theta -1})\\
	&= \frac{1}{2K (1-\theta)} \big((x + K x^\theta)^{1-\theta} - (x - K x^\theta)^{1-\theta} \big) + O(x^{-\theta} + x^{\theta -1}) \\
	&= 1+ O(x^{-\theta}+x^{\theta-1}),
\end{align*}
% Estimating $(x + K x^\theta)^{1-\theta}$ yields
%	\[
%		(x + K x^\theta)^{1-\theta} = x^{1-\theta} \exp \big( (1-\theta) (K x^{\theta-1} + O(x^{2(\theta-1)})) \big) = x^{1-\theta} + (1-\theta) K + O(x^{\theta-1}).
%	\]
%	and similarly $(x - K x^\theta)^{1-\theta} = x^{1-\theta} - (1-\theta) K + O(x^{\theta-1})$. Therefore we have $f(x) = 1 + O(x^{-\theta} + x^{\theta-1})$, so
so we obtain
\begin{equation}\label{eq: integral with f}
	E(S_{J, t}) = \int_{1}^J f(u) u^{-\I t} \dif u + O(J^\theta) = \frac{J^{1 - \I t}}{1 - \I t} + O(J^\theta + J^{1-\theta}).
\end{equation}
%Note that for $\theta < 1/2$ the values at the discontinuities of the function $f(x)$ are of size $x^{-\theta}$, hence it's not straightforward to improve the error-term here.
	
Now we continue as in the proof of Theorem \ref{probabilistic theorem}. Applying Hoeffding's inequality to $\pm \Re(S_{J,t} - E(S_{J,t}))$ and $\pm \Im(S_{J,t} - E(S_{J,t}))$ with $v_{J, t} = 2 \sqrt{J} \bigl(\sqrt{\log(J+1)}+\sqrt{\log(\,\abs{t}+1)}\bigr)$ yields
\[
	P\bigl(\,\abs{S_{J, t} - E(S_{J, t})} \geq \sqrt{2} v_{J, t}\bigr) \le \frac{4}{(J+1)^{2}(\,\abs{t}+1)^{2}}.
\]
Let $A_{J, k}$ denote the event $\abs{S_{J, k} - E(S_{J, k})} \geq \sqrt{2} v_{J, k}$ for integers $J>0$ and $k$. Using the Borel--Cantelli lemma we have that with probability $1$ only finitely many of the events $A_{J, k}$ occur. Now we show that a sequence $\MN = (n_j)_{j\ge1}$ contained in only finitely many sets $A_{J, k}$ satisfies (\ref{eq: probabilistic theorem squareroot}). For $x\ge1$ and $t\in \R$, set $S(x, t) = \sum_{n_j \leq x} n_j^{-\I t}$. For $t=k\in\Z$ and $J=\lfloor x\rfloor$ we have
\begin{align*}
	S(x, k) &= S_{J, k} + O(x^{\theta})\\
	&= E(S_{J},k) + O\bigl(\sqrt{x}\bigl(\sqrt{\log(x+1)}+\sqrt{\log(\,\abs{k}+1)}\bigr)\bigr) + O(x^\theta)\\
	&= \frac{x^{1-\I k}}{1-\I k} + O\bigl(\sqrt{x}\bigl(\sqrt{\log(x+1)}+\sqrt{\log(\,\abs{k}+1)}\bigr)\bigr) + O(x^\theta + x^{1-\theta}),
\end{align*}
where we have used (\ref{eq: integral with f}). For $t \in (k-1, k)$, $k\in \Z$ we apply integration by parts as before to obtain
\[
	S(x, t) = \frac{x^{1-\I t}}{1-\I t} + O\bigl(\sqrt{x}\bigl(\sqrt{\log(x+1)}+\sqrt{\log(\,\abs{t}+1)}\bigr)\bigr) + O(x^\theta + x^{1-\theta}).
\]
This completes the proof.
\end{proof}

One might expect that if $\theta<1/2$, $\LH(\MN, Ax)$ should also hold with probability $1$ as we have good control on the sequence $(n_{j})_{j\ge1}$. However, since the jump discontinuities of the function $f(x)$ in the proof are of size $x^{-\theta}$, it is not clear if the error term $O(J^{1-\theta})$ in \eqref{eq: integral with f} can be improved. If we set $F(x) = \int_{1}^{x}f(u)\dif u$, then all considered sequences have counting function satisfying
\[
	N(x) = F(x) + O(x^{\theta}) = Ax + O(x^{\theta}+x^{1-\theta}),
\]
and $\LH(\MN, F(x))$ holds with probability $1$ provided that $\theta\le 1/2$.

\subsection{Topological viewpoint}
Next we will consider topological spaces of sequences for which the subspace of sequences satisfying LH is \emph{small}. For this we will use the notion of Baire category. We recall that if $X$ is a locally compact Hausdorff space or a complete metric space, then it is a Baire space: the countable union of nowhere dense sets has empty interior. Such countable unions of nowhere dense sets are called \emph{meagre} sets or \emph{sets of first category}, and in Baire spaces they can be considered topologically small or negligible. We refer to \cite{Oxtoby} for more background on Baire category theory, as well as its relation to measure theory.

Given two increasing sequences $(a_{j})_{j\ge1}$ and $(b_{j})_{j\ge1}$ with $a_{j}<b_{j}$, we set 
\[
	X = \prod_{j=1}^{\infty}[a_{j}, b_{j}]= \bigl\{\mathcal{N}=(n_{j})_{j\ge1}: a_{j} \le n_{j} \le b_{j}\bigr\}.
\]
We do not require that $b_{j} \le a_{j+1}$, so the sequences in $X$ need not be non-decreasing. To fix ideas, we suppose that 
\begin{gather}
	a_{j}\sim b_{j} \sim j 	\label{intervals around j}\\
	b_{j}-a_{j} \gg 1 	\label{length intervals bounded below}	
\end{gather}
Assumption \eqref{intervals around j} implies that for any $(n_{j})_{j\ge1} \in X$, its counting function satisfies $N(x) = \sum_{n_{j}\le x}1 \sim x$.

We equip $X$ with two topologies. First we consider $\tau_{\Pi}$, the product topology or the topology of pointwise convergence. Clearly $(X,\tau_{\Pi})$ is a compact Hausdorff space. Second, we consider the normalized uniform metric $d$, defined as
\[
	d(\mathcal{N}_{1},\mathcal{N}_{2}) = \sup_{j\ge1}\frac{\abs{n_{1,j}-n_{2,j}}}{b_{j}-a_{j}}, \quad \text{ where } \mathcal{N}_{i} = (n_{i,j})_{j\ge1}, \quad i=1,2.
\]
It is straightforward to verify that $(X,d)$ is a complete metric space. We denote the associated topology by $\tau_{d}$. Clearly this topology is finer that the product topology: every $\Pi$-open (resp.\ $\Pi$-closed) set is also $d$-open (resp.\ $d$-closed). The nowhere dense sets are not comparable however: there are $\Pi$-nowhere dense sets which are not $d$-nowhere dense, and vice-versa.

It turns out that not only LH is a ``rare'' occurrence in the space $X$, but in fact any asymptotic behavior of the exponential sums $\sum_{n_j\le x} n_{j}^{-\I t}$ is non-typical.
\begin{theorem}
\label{categoric theorem}
Let $F(x,t)$ be a function which is right continuous in $x$. We consider the subspace $Y(F)$ of $X$ of those sequences whose exponential sums behave like $F$:
\begin{multline*}
	Y(F) = \biggl\{ \mathcal{N}\in X: \exists x_{0}>1, \exists 0< c_{1} < 1 < c_{2}, \exists \eta \in [0,1): \\
	\forall x, t \text{ with } x\ge x_{0} \text{ and } x^{c_{1}} \le \abs{t} \le x^{c_{2}}: \abs[3]{\sum_{n_{j}\le x}n_{j}^{-\I t} - F(x,t)} \le \eta x\biggr\}.
\end{multline*}
Then $Y(F)$ is both a $\Pi$-meagre and a $d$-meagre subset of $X$.
\end{theorem}
\begin{proof}
Note that 
\[
	Y(F) = \bigcup_{M>1}\bigcup_{k>1}\bigcup_{l>1}Y(F; M,k,l),
\]
where 
\begin{multline*}
	Y(F;M,k,l) = \\
	\biggl\{\mathcal{N}\in X: \forall x, t \text{ with } x\ge M \text{ and } x^{1-1/k}\le \abs{t} \le x^{1+1/k}:  \abs[3]{\sum_{n_{j}\le x}n_{j}^{-\I t} - F(x,t)} \le (1-1/l) x\biggr\}.
\end{multline*}
It suffices to show that each $Y(F;M,k,l)$ is both $\Pi$-nowhere dense and $d$-nowhere dense. We show something slightly stronger, namely that $Y(F;M,k,l)$ is $\Pi$-closed and has empty $d$-interior. Let $\mathcal{N} = (n_{j})_{j\ge1}$ be an element of the $\Pi$-closure of $Y(F; M, k, l)$, and fix some $x$ and $t$ with $x\ge M$ and $x^{1-1/k} \le \abs{t} \le x^{1+1/k}$. Let $(\mathcal{N}_{m})_{m\ge1}$ be a sequence in $Y(F; M, k, l)$ converging pointwise to $\mathcal{N}$. Let $\eps>0$ be sufficiently small such that $n_{j} > x \implies n_{j}\ge x+2\eps$. Then there is some $m(\eps)$ such that 
\[	
	m\ge m(\eps) \implies \bigl(n_{j}\le x \iff n_{m, j}\le x+\eps\bigr).
\]
Hence
\begin{align*}
	\abs[3]{\sum_{n_{j}\le x}n_{j}^{-\I t} - F(x,t)}	&= \lim_{m\to \infty}\abs[3]{\sum_{n_{j}\le x}n_{m,j}^{-\I t}-F(x,t)}\\
	&=\lim_{m\to\infty}\abs[3]{\sum_{n_{m,j}\le x+\eps}n_{m,j}^{-\I t} - F(x,t)} \le (1-1/l)x + \abs{F(x+\eps,t)-F(x,t)}.
\end{align*}
As $F(x,t)$ is right-continuous in $x$ and $\eps$ can be taken arbitrarily small, $\mathcal{N} \in Y(F;M,k,l)$. This shows that $Y(F;M,k,l)$ is $\Pi$-closed.

Let now $\mathcal{N} \in Y(F;M,k,l)$ and $\eps>0$. We show that $B(\mathcal{N}, \eps) = \{\mathcal{N}'\in X: d(\mathcal{N}, \mathcal{N}')<\eps\}$ is not contained in $Y(F;M,k,l)$. For $x$ and $t$ with $t=x^{1+1/k}$, we have $t \ge 4\pi x/(\eps(b_{j}-a_{j}))$ for all $j$ if $x$ is sufficiently large, in view of \eqref{length intervals bounded below}. For each $j\ge j_{0}$ with $n_{j} \le x$, pick $n_{j}' \in [a_{j}, b_{j}]$ with $\abs[0]{n_{j}-n_{j}'}/(b_{j}-a_{j}) < \eps$ and such that $\arg \bigl((n_{j}')^{-\I t}\bigr) = -\arg F(x,t)$. This is possible, because 
\[
	(n_{j}')^{-\I t} = n_{j}^{-\I t}\exp \Bigl(-\I t\log\frac{n_{j}'}{n_{j}}\Bigr) \quad\text{and}
	\quad \abs[2]{\log\Bigl(\frac{n_{j}\pm\eps(b_{j}-a_{j})}{n_{j}}\Bigr)} > \frac{\eps(b_{j}-a_{j})}{2n_{j}} \ge \frac{2\pi}{t},
\]
if $j_{0}$ is sufficiently large so that $(b_{j}-a_{j})/a_{j} \le 1$ say.
For $j < j_{0}$ or $j$ with $n_{j}>x$, we set $n_{j}'=n_{j}$. Then $\mathcal{N}'=(n_{j}')_{j\ge1} \in B(\mathcal{N}, \eps)$ and $N_{\mathcal{N}'}(x) \sim N_{\mathcal{N}}(x) \sim x$, but
\[
	\abs[3]{\sum_{n_{j}' \le x}(n_{j}')^{-\I t} -F(x,t)} \ge \abs[3]{\sum_{n_{j} \le x}(n_{j}')^{-\I t} - F(x,t)} +o(x) > \Bigl(1-\frac{1}{l}\Bigr)x,
\] 
provided that $x$ is sufficiently large. Hence $\mathcal{N}'\not\in Y(F;M,k,l)$. As $\mathcal{N}$ and $\eps$ were arbitrary, we have shown that $Y(F;M,k,l)$ has empty $d$-interior.
\end{proof}

If we now suppose that $b_{j}-a_{j} \ll_{A} j/(\log j)^{A}$ for all $A>0$, then uniformly for $\MN \in X$ $N_{\MN}(x) - x \ll_{A} x/(\log x)^{A}$ for all $A>0$. If we omit the requirement that $\MN$ should be increasing, then each $\MN \in X$ is admissible with $M(x)=x$. However, the set of sequences $\MN$ satisfying $\LH(\MN, x)$ is clearly contained in $Y(F)$, where $F(x,t) = x^{1-\I t}/(1-\I t)$. In fact, for any fixed $B>1$ and $\eps<1/2$, the set of sequences satisfying
\[
	\forall x,t \text{ with } a_{1}\le x\le \abs{t}^{B}: \abs[3]{\sum_{n_{j}\le x}n_{j}^{-\I t}-\frac{x^{1-\I t}}{1-\I t}} \ll x^{1/2}\abs{t}^{\eps}
\]
is already contained in $Y(F)$. So in the topological sense, the ``generic'' sequence $\MN$ in $X$ does not satisfy $\LH(\MN, x)$.

Of course, many variations of the above construction of topological spaces of sequences are possible. The main philosophy is that any prescribed asymptotic behavior of the exponential sums $\sum_{n_{j}\le x} n_{j}^{-\I t}$ is an unstable property: one can always find small perturbations of the sequence $(n_{j})_{j\ge1}$ which completely destroy this asymptotic behavior.

\bigskip

Suppose now that there is a probability measure $P_{j}$ defined on each interval $[a_{j},b_{j}]$. The Kolmogorov extension theorem then yields a ($\tau_{\Pi}$-)Radon probability measure $P$ on $X$ which extends finite products of the $P_{j}$: for every $J$ and Borel set $A \subseteq \prod_{j=1}^{J}[a_{j},b_{j}]$, 
\[
	P\biggl(A \times \prod_{j=J+1}^{\infty}[a_{j},b_{j}]\biggr) = P_{1}\times \dotsb \times P_{J}(A) = \int_{A}\dif P_{1}(u_{1})\dotsm \dif P_{J}(u_{J}).
\]
The random variables $N_{J}: (n_{j})_{j\ge1} \mapsto n_{J}$ are then independent and have distribution $P_{J}$.
Under suitable conditions, this construction gives rise to a counting function $M(x)$ for which $P$-almost all sequences $\MN \in X$ satisfy $\LH(\MN, M(x))$.

For example, let $a_{j} = j$, $b_{j} = j+1$, $M(x)$ a right-continuous non-decreasing function such that $M(j+1)=j$ and $M(x) < j$ if $x < j+1$. If we set $P_{j} = \dif M \big|_{(j, j+1]}$, then the probability measure $P$ is an explicit manifestation of the probability measure which is implicit in the statement of Theorem \ref{probabilistic theorem}. (The event $\{ (n_{j})_{j\ge1}: n_{j}=a_{j} \text{ for some } j\}$ has probability zero.) Hence Theorem \ref{probabilistic theorem} yields a probability measure on $X$ with respect to which almost all sequences $\MN \in X$ satisfy $\LH(\MN, M(x))$, but Theorem \ref{categoric theorem} with $F(x,t) = \int_{1}^{x}u^{-\I t}\dif M(u)$ shows that in the topological sense, the generic sequence $\MN \in X$ does not satisfy $\LH(\MN, M(x))$.

Similar observations can be made when we set for example $a_{j} = j - Kj^{1/2}$, $b_{j} = j + Kj^{1/2}$ for some $K > 0$ and compare the probabilistic Theorem \ref{probabilistic theorem squareroot} with the Baire-categoric Theorem \ref{categoric theorem} with $F(x,t) = x^{1-\I t}/(1-\I t)$.

\section{The Lindel\"of hypothesis for sequences of Beurling primes}
\label{sec: Beurling}

A Beurling generalized number system $(\MP, \MN)$ consists of a sequence of generalized primes $\MP = (p_{j})_{j\ge1}$, where
\[
	p_{j}\in \R, \quad 1< p_{1} \le p_{2} \le \dotso, \quad p_{j}\to \infty,
\]
and the corresponding sequence of generalized integers $\MN = (n_{j})_{j\ge1}$, which consists of $1=n_{1}$ and all possible (finite) products of generalized primes, ordered in a non-decreasing fashion. Both $\MP$ and $\MN$ are allowed to have repeated values. For example, if $p_{1}=p_{2}=2$, $p_{3}>2$, then the integer $4$ occurs $3$ times in $\MN$, corresponding to $p_{1}^{2}$, $p_{1}p_{2}$, and $p_{2}^{2}$. We denote the counting functions of the sequences by $\pi_{\MP}(x) = \sum_{p_{j}\le x}1$ and $N_{\MP}(x) = \sum_{n_{j}\le x}1$ respectively, and omit the subscript $\MP$ if there is no risk of confusion. We also consider the Chebyshev function, defined as $\psi_{\MP}(x) = \sum_{p_{j}^{\nu}\le x}\log p_{j}$.

Such generalized number systems were introduced by Beurling \cite{Beurling}, who wanted to investigate the ``stability'' of theorems from multiplicative number theory. For example, he showed that the prime number theorem (PNT) $\pi(x) \sim x/\log x$ follows from the estimate $N(x) = Ax + O(x/\log^{\gamma}x)$ for some $A>0$, provided that $\gamma>3/2$. On the other hand, the Riemann hypothesis need not hold in this general setting, as Diamond, Montgomery, and Vorhauer \cite{DiamondMontgomeryVorhauer} (see also \cite{Zhang}) established the existence of generalized number systems satisfying $N(x)=Ax + O_{\eps}(x^{1/2+\eps})$ for every $\eps>0$, but for which the classical de la Vall\'ee-Poussin estimate is best possible: $\pi(x) = \Li(x)  +\Omega\bigl(x\exp(-c\sqrt{\log x})\bigr)$ for some $c>0$.

Given a generalized number system $(\MP, \MN)$, we associate with it the Beurling zeta function $\zeta_{\MP}(s) \coloneqq \sum_{n_{j}}n_{j}^{-s}$. The multiplicative structure of the Beurling integers is encoded in the Euler product
\[
	\zeta_{\MP}(s) = \prod_{p_{j}}\frac{1}{1-p_{j}^{-s}},
\]
from which it follows that 
\[
	-\frac{\zeta_{\MP}'(s)}{\zeta_{\MP}(s)} = \int_{1}^{\infty}x^{-s}\dif \psi_{\MP}(x).
\]
For a general introduction to the theory of Beurling number systems, we refer to \cite{DiamondZhangbook}.

In \cite{GGL}, Gonek, Graham, and Lee showed Theorem \ref{GGL theorem}, stating that the Riemann hypothesis is equivalent to $\LH(\mathbb{P}, \Li(x))$. Here we extend this result to the context of Beurling generalized numbers. We consider number systems $(\MP, \MN)$ for which $N_{\MP}(x) = Ax + O(x^{\theta})$ for some $A>0$ and $\theta \in (0,1)$. In particular they satisfy the PNT.
We recall that we use $s=\sigma+\I\tau$ to denote a complex variable, while $t$ will denote the variable occurring in the exponent in the exponential sums related to $\LH$. 
We first define for $x\ge1$ and real $t$
\begin{align*}
	\pi(x,t) 	&= \sum_{p_{j}\le x}p_{j}^{-\I t},&				r(x,t) 		&= \pi(x,t) - \int_{p_{1}}^{x}\frac{u^{-\I t}}{\log u}\dif u;\\ 
	\psi(x,t) 	&= \sum_{n_{j}\le x}\Lambda(n_{j})n_{j}^{-\I t},&		R(x,t) 	&= \psi(x,t) - \frac{x^{1-\I t}}{1-\I t}.
\end{align*}
For $0< x < 1$ we set $r(x,t)=R(x,t) = 0$.
Then $\LH(\MP, \Li(x))$ is equivalent to the statement that for every $\eps>0$ and $B>0$,
\begin{equation}
\label{eq: LH(P) with Lambda}
	R(x,t) \ll_{\eps, B} x^{1/2}\abs{t}^{\eps}, \quad \text{for } 1 \le x \le \abs{t}^{B}.
\end{equation}
Indeed, writing
\[
	\psi(x,t) = \int_{1}^{x}\log u\dif \pi(u, t) + O(x^{1/2}), \quad \pi(x,t) = \int_{1}^{x}\frac{\dif\psi(u,t)}{\log u} + O\Bigl(\frac{x^{1/2}}{\log x}\Bigr),
\]
integrating by parts, and using that $\log x \ll_{\eps, B} \abs{t}^{\eps}$ and the fact that $\pi(x) \ll x/\log x$, it is clear that $r(x, t) \ll \Li(x)^{1/2}\abs{t}^{\eps}$ is equivalent to $R(x,t) \ll x^{1/2}\abs{t}^{\eps}$.

\bigskip

Before giving the proof of Theorem \ref{th: LH(P) <=> RH}, we need a couple of preparatory lemmas. First we state a well-known effective Perron inversion formula.

Let $(n_{j})_{j\ge1}$ be a non-decreasing unbounded sequence with $n_{1}\ge1$ and $(a_{j})_{j\ge1}$ be an arbitrary complex sequence. Write $A(x) = \sum_{n_{j}\le x}a_{j}$, and suppose $F(s) = \sum_{j\ge1}a_{j}n_{j}^{-s}$ converges absolutely for $\sigma > \sigma_{a}$. Then for $\kappa > \max(0,\sigma_{a})$, $T\ge1$, and $x\ge1$,  
\begin{equation}
\label{effective Perron}
	A(x) = \frac{1}{2\pi\I}\int_{\kappa-\I T}^{\kappa+\I T}F(s)x^{s}\frac{\dif s}{s} + O\biggl(x^{\kappa}\sum_{j\ge1}\frac{\abs[0]{a_{j}}}{n_{j}^{\kappa}\bigl(1+T\abs[0]{\log(x/n_{j})}\bigr)}\biggr).
\end{equation}
See e.g.\ \cite[Theorem II.2.3]{Tenenbaum} for the case $n_{j} = j$ (the proof of the general case is identical).

\begin{lemma}
\label{lem: convexity bounds}
Suppose $N(x) = Ax + O(x^{1/2})$ for some $A>0$. Then $\zeta(s)$ has analytic continuation to the half-plane $\Re s > 1/2$, except for a simple pole at $s=1$ with residue $A$. %If $N(x) = Ax + O(x^{1/2})$, then 
Furthermore $\zeta(s)$ satisfies the following bounds (where $s=\sigma+\I \tau$):
\begin{align*}
	\zeta(s) 	&\ll \Bigl(\frac{\abs{\tau}+1}{\sigma-1/2}\Bigr)^{2-2\sigma}, 	&&\text{for } \frac{1}{2} < \sigma < \frac{3}{4}, \\
	\zeta(s)	&\ll \abs{\tau}^{2-2\sigma}\log\abs{\tau}+1,				&&\text{for } \sigma \ge \frac{3}{4}, \quad \abs{\tau} \ge 2.
\end{align*}
\end{lemma}
\begin{proof} We write \( N(x) = Ax + E(x) \) with \( E(x) \ll x^{1/2} \). Then we have for $\sigma > 1$ and $x > 1$
\begin{align*}
	\zeta(s) &= \sum_{n_j \leq x} n_j^{-s} + \int_x^\infty \frac{\dif N(u)}{u^s} 
	= \sum_{n_j \leq x} n_j^{-s} + A \int_x^\infty  \frac{\dif u}{u^s} + \int_x^\infty \frac{\dif E(u)}{u^s}   \\
	&= \sum_{n_j \leq x} n_j^{-s} + A\frac{x^{1-s}}{s-1} - \frac{E(x)}{x^s} + s \int_x^\infty \frac{E(u)}{u^{s+1}} \dif u.
\end{align*}
As $E(x) \ll x^{1/2}$ this has, for fixed $x$, analytic continuation to $\sigma > 1/2$, except for a simple pole at $s=1$ with residue $A$. Consequently, for $\sigma > 1/2$ we have
\[
	\abs{\zeta(s) - \frac{A}{s-1}} \ll \sum_{n_{j}\le x}n_{j}^{-\sigma} + \abs{\frac{x^{1-s}-1}{1-s}} + \abs{s}\frac{x^{1/2-\sigma}}{\sigma-1/2}.
\]
Using again $N(x) = Ax + O(x^{1/2})$ we obtain for $\sigma>1/2$
\[
	\sum_{n_{j}\le x}n_{j}^{-\sigma} = A\frac{x^{1-\sigma}-1}{1-\sigma} + O\Bigl(\frac{\sigma}{\sigma-1/2}\Bigr),
\]
where we take the first term of the right hand side to be $A\log x$ if $\sigma=1$.

If $1/2<\sigma<3/4$ we take $x = \bigl(\max\{1,|\tau|\}/(1/2-\sigma)\bigr)^{2}$, yielding 
\[
	\abs{\zeta(s) - \frac{A}{s-1}} \ll \Bigl(\frac{\abs{\tau}+1}{\sigma-1/2}\Bigr)^{2-2\sigma}.
\]
If $\sigma\ge3/4$ and $\abs{\tau}\ge2$ we take $x = \abs{\tau}^{2}\log\abs{\tau}$, yielding
\[
	\zeta(s)	\ll \abs{\tau}^{2-2\sigma}\log\abs{\tau}+1.
\]
\end{proof}

\begin{lemma}
\label{lem: bounds zeta'/zeta}
Suppose $N(x) = Ax + O(x^{1/2})$ for some $A>0$, and that $\zeta(s)$ has no zeros in $\sigma > 1/2$. Then $-\zeta'(s)/\zeta(s)$ has meromorphic continuation to $\sigma > 1/2$, with a simple pole at $s=1$ with residue $1$. For each $\delta>0$, we have
\[
	-\frac{\zeta'(s)}{\zeta(s)} \ll \frac{1}{\delta^{2}}\Bigl(\log(|\tau|+2) + \log\frac{1}{\delta}\Bigr), \quad \text{for } \sigma\ge \frac{1}{2}+\delta \text{ and } \abs{s-1} \ge \frac{1}{4}.
\]
\end{lemma}
\begin{proof}
As $\zeta(s)$ has no zeros for $\sigma > 1/2$, $\log(\zeta(s)\frac{s-1}{s})$ defines an analytic function in this half-plane. From Lemma \ref{lem: convexity bounds} it easily follows that
\begin{equation}\label{eq: bound log zeta}
	\log \abs{\zeta(s)\frac{s-1}{s}} \le \log (|\tau| + 2)+ \log\frac{1}{\delta} + O(1), \quad \text{for } \sigma \geq \frac{1}{2} + \frac{\delta}{4}.
\end{equation}
Let $z=2+\I \tau$. Applying the Borel--Carath\'eodory lemma %(see e.g. \cite[Section 5.5]{Titchmarsh1939}) 
%\rge{[I don't think it's necessary to provide a reference here]}
to the two circles with centre $z$ and radii $R = 3/2 - \delta/4$  and $r = 3/2 - \delta/2$ yields 
\[
	\max_{\abs{s - z} = r} \abs{\log \Bigl(\zeta(s)\frac{s-1}{s}\Bigr)} \le \frac{2r}{R-r} \max_{\abs{s - z} = R} \log \abs{\zeta(s)\frac{s-1}{s}} + \frac{R+r}{R-r} \abs{\log \Bigl(\zeta(z)\frac{z-1}{z}\Bigr)}.
\]
Using \eqref{eq: bound log zeta} it follows that
\[
	\log \Bigl(\zeta(s)\frac{s-1}{s}\Bigr) \ll \frac{1}{\delta}\Bigl(\log(|\tau|+2) + \log\frac{1}{\delta}\Bigr), \quad \text{for } \sigma \geq \frac{1}{2} + \frac{\delta}{2}.
\]
By Cauchy's formula we obtain
\[
	\frac{\zeta'(s)}{\zeta(s)} = \frac{1}{2\pi \I} \oint \frac{\log(\zeta(z)\frac{z-1}{z})}{(z-s)^2} \dif z - \frac{1}{s-1} + \frac{1}{s} \ll \frac{1}{\delta^{2}}\Bigl(\log(|\tau|+2) + \log\frac{1}{\delta}\Bigr),
\]
 for $\sigma \geq 1/2 + \delta$ and $\abs{s-1} \geq 1/4$ say, and where the integral is over a circle with center $s$ and radius $\delta/2$.
\end{proof}

Recall that we have said that the Riemann hypothesis holds for $(\MP, \MN)$ if $\psi(x)=x+O_{\eps}(x^{1/2+\eps})$ for all $\eps>0$. Using the above lemma and Perron inversion, it now easily follows that $N(x) = Ax+O(x^{1/2})$ and the non-vanishing of $\zeta(s)$ for $\Re s >1/2$ imply RH (indeed, one can even get $\psi(x) = x + O(x^{1/2}(\log x)^{4})$). Without assuming $N(x)=Ax+O(x^{1/2})$, merely the analyticity of $-\zeta'(s)/\zeta(s)$ for $\Re s > 1/2$ does not imply RH, as $-\zeta'(s)/\zeta(s)$ may have extreme growth.

On the other hand, RH implies of course that $\zeta(s)$ has analytic continuation to $\Re s > 1/2$ with the exception of a simple pole at $s=1$, and that it does not vanish there.

\bigskip
We are now ready to give the proof of Theorem \ref{th: LH(P) <=> RH}. %We recall again that $t$ does not denote the imaginary part of the complex variable $s=\sigma+\I\tau$, but refers to the variable in the formulation of $\LH$.
Suppose first that the Riemann hypothesis holds. 
Let $t$ be real with $|t| \geq 1$ and let $m$ be a positive integer with $m > \e^4$. Since $N(x) = Ax + O(x^{1/2})$ by assumption there are $N(m+1) - N(m) = O(m^{1/2})$ Beurling integers in the interval $(m, m+1]$. Hence we can choose $x \in (m, m+1]$ such that $|x - n_j| \gg x^{-1/2}$ for all $n_j$. Furthermore let $\kappa = 1 + 1/\log x$ and $T \geq |t| + 2$. Then applying the effective Perron inversion formula \eqref{effective Perron} with $a_{j} = \Lambda(n_{j})n_{j}^{-\I t}$ yields 
\begin{align*}
		\psi(x, t) = \sum_{n_{j}\le x}\Lambda(n_{j})n_{j}^{-\I t} 
		&= -	\frac{1}{2\pi \I } \int_{\kappa - \I T}^{\kappa + \I T} \frac{\zeta'(s+ \I t)}{\zeta(s + \I t)} x^{s}\frac{\dif s}{s} \\
		&+ O\Big( -\frac{\zeta'(\kappa)}{\zeta(\kappa)}\frac{x}{T} \Big) 
		+ O\Big( \frac{x \log x}{T} \sum_{x/2 < n_j \leq 2x} \frac{1}{\abs{x - n_j}} \Big).
\end{align*} 
Using $-\zeta'(\sigma)/\zeta(\sigma) \ll 1/(\sigma-1)$ for $\sigma\to1^{+}$ and $|x - n_j| \gg x^{-1/2}$ we obtain %\rge{[you can even get $O(x^{2}\log x/T)$ by estimating the sum by partial summation.]}
\begin{equation}\label{eq: integral psi}
	\psi(x, t) = - \frac{1}{2\pi \I} \int_{\kappa - \I T}^{\kappa + \I T} \frac{\zeta'(s + \I t)}{\zeta(s + \I t)} x^{s}\frac{\dif s}{s} + O \Big( \frac{x^{5/2} \log x}{T} \Big).
\end{equation}
In order to evaluate the integral on the right, we integrate along the boundary of the rectangle with vertices $\kappa-\I T, \kappa+\I T, d+\I T, d-\I T$, where $d = 1/2 + 1/\log x < 3/4$ by the assumption on $m$. 
The function $\zeta'(s)/\zeta(s)$ is meromorphic in the right half-plane $\sigma > 1/2$ with a simple pole at $s = 1$ of residue $-1$. Since the Riemann hypothesis holds by assumption, there are in fact no further poles in this right half-plane. Hence
\begin{equation*}
	\frac{1}{2\pi \I } \int_{\kappa - \I T}^{\kappa + \I T} \frac{\zeta'(s + \I t)}{\zeta(s + \I t)} x^{s}\frac{\dif s}{s} 
	= - \frac{x^{1-\I t}}{1 - \I t} + \frac{1}{2\pi \I } \Big(\int_{\kappa - \I T}^{d - \I T} +  \int_{d - \I T}^{d + \I T} + \int_{d + \I T}^{\kappa + \I T} \Big) \frac{\zeta'(s + \I t)}{\zeta(s + \I t)} x^{s}\frac{\dif s}{s}.
\end{equation*}
Applying Lemma \ref{lem: bounds zeta'/zeta} gives
\begin{align*}
	\int_{d \pm \I T}^{\kappa \pm \I T} \frac{\zeta'(s + \I t)}{\zeta(s + \I t)} x^s\frac{\dif s}{s} 	
		&\ll (\log T)(\log x)^3 \int_d^\kappa \frac{x^\sigma}{|\sigma + \I T|} \dif \sigma \ll \frac{x (\log T)(\log x)^2}{T}, \\
	\int_{d-iT}^{d+iT} \frac{\zeta'(s + \I t)}{\zeta(s + \I t)}x^s\frac{\dif s}{s} 				
		&\ll x^{d}(\log T)(\log x)^3 \int_{-T}^{T} \frac{\dif \tau}{d + \I \tau} \ll x^{1/2}(\log T)^{2}(\log x)^{3}, 
\end{align*}
where we have used $T \geq |t| + 2$. 
Collecting all estimates yields
\[
	\psi(x, t) = \frac{x^{1-\I t}}{1 - \I t} + O\Big( \frac{x^{5/2} \log x}{T}  + x^{1/2}(\log T)^{2}(\log x)^{3}\Big).
\]
Let now $B > 0$ be arbitrary and suppose $\e^4 + 1 \leq x \leq |t|^B$. We then choose $ T = \max(|t|^{2B}, |t| + 2) \geq x^2$. Hence it follows that
\begin{equation}\label{eq: psi end}
	\psi(x, t) = \sum_{n_{j}\le x}\Lambda(n_{j})n_{j}^{-\I t} = \frac{x^{1- \I t}}{1 - \I t} + O_{\eps, B}(x^{1/2}|t|^\eps).
\end{equation}
In the beginning of the proof we restricted $x$ to be a specific number in the interval $(m, m+1]$. Now if we replace $x$ by an arbitary $\tilde{x} \in (m, m+1]$, the left-hand side of \eqref{eq: psi end} changes by at most $O(x^{1/2+\eps}) = O(x^{1/2}|t|^{\eps})$, while the main term of the right-hand side changes by at most $O(1)$. Hence \eqref{eq: psi end} holds for all $ \e^4 + 1 \leq x \leq |t|^B $ and also trivially for $1 \leq x < \e^4 + 1$. This proves $\LH(\MP, \Li(x))$.

%\bl{Gonek et. al gets an error of size $|t|^\eps$ on the right-hand side - how? I think it should be $O(1)$ with distinguishing the two cases $t = o(x)$ and $t \gg x$ (then the dominant term itself is $O(1)$).}
%\rge{[I also don't know why they write $|t|^{\eps}$. You don't even need to distinguish cases: $\frac{x^{1-\I t}-\tilde{x}^{1-\I t}}{1-\I t} = \int_{\tilde{x}}^{x}u^{-\I t}\dif u \ll 1$.]}

\bigskip

Suppose now that $\LH(\MP, \Li(x))$ holds. One might think that one can show that $\zeta(s)$ does not have zeros in the half-plane $\Re s > 1/2$ based on some kind of explicit formula
\[
	R(x,t) = -\sum_{\rho}\frac{x^{\rho-\I t}}{\rho-\I t} + \text{error},
\]
where the sum is over zeros $\rho=\beta+\I \gamma$ of $\zeta(s)$, limited to certain suitable ranges for $\beta$ and $\gamma$. (Actually, such an explicit formula for $\psi(x)$, or equivalently, $R(x,0)$, was recently shown in the Beurling setting by R\'ev\'esz \cite[Theorem 5.1]{Revesz2021}\footnote{A term $x/t_{k}$ is missing in the error term in the statement of \cite[Theorem 5.1]{Revesz2021}}.) A zero $\rho=\beta+\I\gamma$ contributes to a term of size $x^{\beta}$ in the sum (if the term $\I t$ is sufficiently close to $\rho$ say). Hence the existence of a zero $\rho_{0} = \beta_{0} + \I\gamma_{0}$ with $\beta_{0}> 1/2$ would contradict $\LH(\MP, \Li(x))$. The problem is that it is very difficult to establish that there is no too severe cancellation in the sum caused by other zeros, even already in the case of the Riemann zeta function and with $t=0$ (see e.g.\ \cite{Pintz1980}).

To show that $\zeta(s)$ has no zeros in $\Re s>1/2$ we instead follow a more subtle approach of Gonek, Graham, and Lee, based on a method by Pintz (see \cite{Pintz1982}). The idea is to consider a suitable average of $R$ which suppresses the influence of all but one zero (namely the zero $\rho_{0}$ with $\beta_{0}>1/2$).

We start by setting
\[
	H_{t}(s) \coloneqq \int_{1^{-}}^{\infty}x^{-s}\dif R(x,t) = -\frac{\zeta'(s+\I t)}{\zeta(s+\I t)} - \frac{s}{(1-\I t)(s+\I t-1)}.
\]
for $t\in \R$. This function has meromorphic continuation to $\sigma >1/2$, with poles at the points $s=\rho-\I t$ for each zero $\rho=\beta+\I \gamma$ of $\zeta(s)$ with $\beta>1/2$. Suppose now that $\rho_{0}=\beta_{0}+\I\gamma_{0}$ with $\beta_{0}>1/2$ is such a zero, and let $m$ be its multiplicity. We then consider
\[
	h_{t}(s) \coloneqq \frac{\zeta(s+\I t)(s+\I t-1)}{(s+\I t-\rho_{0})^{m}(s+\I t)^{3}},
\]
which defines a function analytic in $\sigma > 1/2$.  Denote its inverse Mellin transform by $W_{t}(y)$, that is
\[
	W_{t}(y) \coloneqq \frac{1}{2\pi\I}\int_{2-\I\infty}^{2+\I\infty}h_{t}(s)y^{s}\frac{\dif s}{s}, \quad y > 0.
\]
Since $\zeta(s)$ is bounded on the half-plane $\sigma \ge2$, the integral converges absolutely and locally uniformly in $y$, so that $W_{t}(y)$ is a continuous function. It is furthermore supported on $[1, \infty)$, as can be seen by shifting the contour to the right when $0<y\le 1$. 
The function $W_{t}$ is the weight by which we shall average $R$. Indeed, we set
\[
	I_{t}(x) \coloneqq \int_{1^{-}}^{x}\dif R(\cdot, t)\ast \dif W_{t} = \int_{1^{-}}^{x}R\Bigl(\frac{x}{y},t\Bigr)\dif W_{t}(y) = \int_{1^{-}}^{x}W_{t}\Bigl(\frac{x}{y}\Bigr)\dif R(y,t).
\]
Here, $\ast$ denotes the multiplicative convolution of measures, a notion extending Dirichlet convolution (see e.g.\ \cite[Chapter 2]{DiamondZhangbook}). Now $W_{t}(y)$ is differentiable, with
\begin{equation}
\label{eq: Perron W'}
	W_{t}'(y) = \frac{1}{2\pi\I}\int_{2-\I\infty}^{2+\I\infty}h_{t}(s)y^{s-1}\dif s.	
\end{equation}
Hence we have
\begin{equation}
\label{eq: I with W'}
	I_{t}(x) = \int_{1}^{x}R\Bigl(\frac{x}{y},t\Bigr)W_{t}'(y)\dif y.
\end{equation}

First we bound $I_{t}(x)$ by employing the bounds provided by $\LH(\MP, \Li(x))$ and by bounding $W_{t}'(y)$. Let $x$ be sufficiently large so that $1/\log x < \frac{\beta_{0}-1/2}{2}$. In the Perron integral \eqref{eq: Perron W'} for $W_{t}'$, we shift the contour of integration to the line $\sigma = 1/2 + 1/\log x$. The convexity bound for $\zeta(s)$ provided by Lemma \ref{lem: convexity bounds} yields
\[
	\zeta\Bigl(\frac{1}{2} + \frac{1}{\log x} + \I \tau\Bigr) \ll (|\tau|+1)\log x,
\]	
hence we get
\[
	W_{t}'(y) \ll y^{-1/2+\frac{1}{\log x}}\int_{-\infty}^{\infty}\frac{\log x}{|1/2+1/\log x + \I \tau - \rho_{0}|^{m}(|\tau|+1)}\dif\tau \ll_{\rho_{0}}y^{-1/2+\frac{1}{\log x}}\log x.
\]
Substituting this bound and the Lindel\"of hypothesis for $\MP$ in the form \eqref{eq: LH(P) with Lambda} into \eqref{eq: I with W'} yields
\begin{equation}
\label{eq: bound I}
	I_{t}(x) \ll_{\eps, B, \rho_{0}} \int_{1}^{x}(x/y)^{1/2}\abs{t}^{\eps}y^{-1/2}\log x\dif y = x^{1/2}(\log x)^{2}\abs{t}^{\eps},
\end{equation}
provided that $1\le x\le \abs{t}^{B}$ and $1/\log x < \frac{\beta_{0}-1/2}{2}$.
	
On the other hand, 
\[
	I_{t}(x) = \int_{1^{-}}^{\infty}W_{t}\Bigl(\frac{x}{y}\Bigr)\dif R(y, t) = \frac{1}{2\pi\I}\int_{2-\I\infty}^{2+\I\infty}h_{t}(s)H_{t}(s)x^{s}\frac{\dif s}{s},
\]
by inverting the order of integration by Fubini's theorem (use that $\abs[0]{\dif R(y,t)} \le \dif\psi(y) + \dif y$). We have
\begin{align*}
	h_{t}(s)H_{t}(s) 	&= \frac{\zeta(s+\I t)(s+\I t-1)}{(s+\I t-\rho_{0})^{m}(s+\I t)^{3}}\biggl(-\frac{\zeta'(s+\I t)}{\zeta(s+\I t)} - \frac{s}{(1-\I t)(s+\I t -1)}\biggr) \\
				&= -\frac{\zeta'(s+\I t)(s+\I t-1)}{(s+\I t-\rho_{0})^{m}(s+\I t)^{3}} - \frac{\zeta(s+\I t)s}{(1-\I t)(s+\I t-\rho_{0})^{m}(s+\I t)^{3}}.
\end{align*}
This function is analytic for $\sigma > 1/2$ except for a simple pole at $s=\rho_{0}-\I t$ with residue 
\[
	-\frac{\zeta^{(m)}(\rho_{0})(\rho_{0}-1)}{(m-1)!\rho_{0}^{3}}.
\]
Shifting the contour of integration again to $\sigma = 1/2 + 1/\log x$, we get
\begin{equation}
\label{eq: evaluation I}
	I_{t}(x) = -\frac{\zeta^{(m)}(\rho_{0})(\rho_{0}-1)}{(m-1)!\rho_{0}^{3}} \cdot \frac{x^{\rho_{0}-\I t}}{\rho_{0}-\I t} + O_{\rho_{0}}\bigl(x^{1/2}(\log x)^{2}\bigr).
\end{equation}
Here we used Cauchy's formula and Lemma \ref{lem: convexity bounds} giving 
\[
	\zeta'\Bigl(\frac{1}{2} + \frac{1}{\log x} + \I \tau\Bigr) = \frac{1}{2\pi\I}\oint\frac{\zeta(z)}{\bigl(z-(1/2+1/\log x+\I \tau)\bigr)^{2}}\dif z \ll (|\tau|+1)(\log x)^{2},
\]
where the integral is over a circle with centre $1/2+1/\log x+\I \tau$ and radius $1/(2\log x)$ say. Combining \eqref{eq: bound I} and \eqref{eq: evaluation I} gives
\[
	\abs[3]{\frac{x^{\rho_{0}-\I t}}{\rho_{0}-\I t}} \ll_{\eps, B, \rho_{0}} x^{1/2}(\log x)^{2}\abs{t}^{\eps}, \quad \text{for } 1\le x\le \abs{t}^{B}, \quad \frac{1}{\log x} < \frac{\beta_{0}-1/2}{2}.
\]
Letting $t= x^{1/B}$ we get
\[
	x^{\beta_{0}-1/2} \ll (\log x)^{2}x^{\frac{1+\eps}{B}}, \quad x\to \infty
\]
which leads to a contradiction upon selecting $B > \frac{1+\eps}{\beta_{0}-1/2}$.

In conclusion, $\LH(\MP,\Li(x))$ implies that $\zeta(s)$ has no zeros in $\sigma > 1/2$.
\bigskip

In Theorem \ref{th: LH(P) <=> RH} we made the assumption that $N(x) = Ax + O(x^{1/2})$. Besides yielding the analytic continuation of $\zeta(s)$ to $\Re s >1/2$, this assumption also provides polynomial bounds on the zeta function (see Lemma \ref{lem: convexity bounds}), which are crucial in the proof of Theorem \ref{th: LH(P) <=> RH}. In the absence of such bounds, $\LH(\MP, \Li(x))$ may fail, even when $\zeta(s)$ has no zeros in $\Re s > 1/2$, as the next propositions shows.

\begin{proposition}
\label{counterexample weakening square root error term}
For every $\beta \in (1/2,1)$ there exists a Beurling number system $(\MP, \MN)$ for which $N(x) = Ax + O(x^{\beta})$ and $\psi(x) = x + O(\log x\log\log x)$, yet for which $\LH(\MP, \Li(x))$ is false.
\end{proposition}

\begin{proof}
%\bl{[At this moment, I am also writing a paper on examples of so-called ``well-behaved'' Beurling number systems. This example, i.e. with $\psi(x) = x + O(\log x\log\log x)$ and large growth of $\zeta(s)$ is a special case of one of the constructions from the other paper. I am inclined not the include all the details here, because it could take two or three pages and deter from the main points of our paper. At the moment, I only added a sketch]

The idea is to construct a number system $(\MP,\MN)$ satisfying $\psi(x) = x + O(\log x\log\log x)$, and whose zeta function $\zeta(s)$ has the following growth properties.
For every $\eps>0$,
\begin{equation}
\label{eq: upper bound zeta}
	\zeta(\sigma+\I \tau) \ll_{\eps} \abs{\tau}^{\eps}, \quad\text{for } \sigma\ge\beta, \quad \abs{\tau}\ge T_{0}(\eps);
\end{equation}
there exists a sequence of ordinates $(\tau_{k})_{k\ge1}$ increasing to $\infty$ such that, for each fixed $\sigma_{0}$ and $\sigma_{1}$ with $1/2 < \sigma_{0} < \sigma_{1} < \beta$, and for every $\eps>0$,
\begin{equation}
\label{eq: lower bound zeta}
	-\frac{\zeta'(\sigma+\I\tau_{k})}{\zeta(\sigma+\I\tau_{k})} \gg \tau_{k}^{\beta-\sigma-\eps}, \quad \text{for }\sigma_{0} \le \sigma \le \sigma_{1},
\end{equation}
provided that $k\ge K$ for some $K=K(\eps)$. 

The upper bound \eqref{eq: upper bound zeta} implies by Perron inversion that $N(x) = Ax + O_{\eps}(x^{\beta+\eps})$ for some $A>0$ and every $\eps>0$ (so to get the bound $O(x^{\beta})$, we apply the construction with a slightly smaller $\beta'$, $1/2<\beta'<\beta$).

Next we show that the lower bound \eqref{eq: lower bound zeta} is incompatible with $\LH(\MP, \Li(x))$ (see also Theorem \ref{LH <=> zeta bound} below).
For real $t$, write again 
\[
	\psi(x,t) = \sum_{n_{j}\le x}\Lambda(n_{j})n_{j}^{-\I t} = \frac{x^{1-\I t}}{1-\I t} + R(x,t).
\] 
We can approximate $-\zeta'(s)/\zeta(s)$ by the partial sums of its defining Dirichlet series even when $\sigma < 1$. 
%If $\Re s >1$ and $X\ge 1$, then
%\begin{align*}
%	-\frac{\zeta'(s)}{\zeta(s)} 	
%		&= \sum_{n_{j}\in \MN}\frac{\Lambda(n_{j})}{n_{j}^{s}} = \sum_{n_{j} \le X}\frac{\Lambda(n_{j})}{n_{j}^{s}} + \int_{X}^{\infty}x^{-s}\dif \psi(x) \\
%		&= \sum_{n_{j} \le X}\frac{\Lambda(n_{j})}{n_{j}^{s}} - \frac{X^{1-s}}{1-s} + \int_{X}^{\infty}x^{-s}\dif R(x,0).
%\end{align*}
%In view of the bound $R(x,0) \ll \log x\log\log x$, both sides have analytic continuation to $\Re s >0$ (apart from the pole at $s=1$). Integration by parts yields
Indeed, following the proof idea of Lemma \ref{lem: convexity bounds} and using $R(x, 0) \ll_{\eps} x^{\eps}$ for all $\eps>0$ we obtain for fixed $x\ge1$
\begin{equation}
\label{eq: approx by partial sum}
	-\frac{\zeta'(s)}{\zeta(s)} = 	\sum_{n_{j} \le x}\frac{\Lambda(n_{j})}{n_{j}^{s}}  - \frac{x^{1-s}}{1-s} + O_{\eps}\Bigl(\,\abs{\tau}\frac{x^{\eps-\sigma}}{\sigma-\eps}\Bigr), 
	\quad \text{for } \abs{\tau}\ge1 \text{ and } \sigma > \eps,
\end{equation}
for every $\eps>0$. Suppose now that $\LH(\MP, \Li(x))$ holds. Then we can use \eqref{eq: LH(P) with Lambda} to estimate the first sum: if $1\le x\le \abs{\tau}^{B}$ for some fixed $B$, then 
\begin{align*}
	\sum_{n_{j} \le x}\frac{\Lambda(n_{j})}{n_{j}^{s}} 	
		&= \int_{1}^{x}u^{-\sigma}\dif\psi(u,\tau) = \int_{1}^{x}u^{-\sigma-\I\tau}\dif u + \int_{1}^{x}u^{-\sigma}\dif R(u,\tau) \\
		&= \frac{x^{1-s}-1}{1-s} + O\Bigl(\, \abs{\tau}^{\eps}\frac{x^{1/2-\sigma}-1}{1/2-\sigma}\Bigr).
\end{align*}
Inserting this into \eqref{eq: approx by partial sum} and selecting $B=2$, $x=\abs{\tau}^{2}$, $\sigma>1/2$ yields
\[
	-\frac{\zeta'(s)}{\zeta(s)} \ll_{\eps} \abs{\tau}^{\eps},
\]
for every $\eps>0$. Fixing some $\sigma\in [\sigma_{0},\sigma_{1}]$, this contradicts \eqref{eq: lower bound zeta} upon selecting $\eps$ small enough.

Finally we indicate how to construct a number system with the desired properties. In fact, this construction is a special case of a construction carried out in an upcoming paper by the first author together with G.\ Debruyne and Sz.\ R\'ev\'esz, %\cite{BrouckeDebruyneRevesz}
so we limit ourselves to mentioning the main ideas. The starting point is to consider a template Chebyshev ``prime-counting'' function $\psi_{C}$, with associated zeta function $\zeta_{C}(s)$ which has the desired properties. Namely, we consider a rapidly increasing sequence $(\tau_{k})_{k\ge1}$ and set
\[
	\psi_{C}(x) = x - \log x - 1 + \sum_{k=1}^{\infty}R_{k}(x),
\]
where $R_{k}(x) = \int_{1}^{x}\dif R_{k}(u)$ with  
\[
	\dif R_{k}(u) = \begin{dcases}
		\tau_{k}\cos(\tau_{k}\log u)u^{\beta-2}\dif u, 	&\text{if } A_{k}\le u< B_{k}, \\
		0									&\text{else.}
	\end{dcases}
\]
Here, $A_{k}=\tau_{k}^{1+\delta_{k}}$, and $B_{k}=\tau_{k}^{\nu_{k}}$ for certain sequences $\delta_{k}\to 0$ and $\nu_{k}$ close to 2, chosen in such a way that $\tau_{k}\log A_{k}, \tau_{k}\log B_{k}\in 2\pi\Z$. One readily verifies that $\sum_{k}R_{k}(x) \ll 1$, and that 
\[
	-\frac{\zeta_{C}(s)}{\zeta_{C}(s)} \coloneqq \int_{1}^{\infty}x^{-s}\dif \psi_{C}(x) = \frac{1}{s-1}-\frac{1}{s} + \sum_{k=1}^{\infty}\eta_{k}(s),
\]
where 
\[
	\eta_{k}(s) = \frac{\tau_{k}}{2}\bigl(A_{k}^{\beta-1-s}-B_{k}^{\beta-1-s}\bigr)\biggl(\frac{1}{\beta-1-s+\I\tau_{k}} + \frac{1}{\beta-1-s-\I\tau_{k}}\biggr).
\]
The zeta function $\zeta_{C}(s)$ satisfies the desired bounds. 

To finish the construction, one applies a ``discretization procedure'' such as Theorem 1.2 of \cite{BrouckeVindas} (or Theorem \ref{probabilistic theorem} of this paper) to show that there exists a Beurling number system $(\MP, \MN)$ for which $\psi_{\MP}(x)$ and $\zeta_{\MP}(s)$ are suitably approximated by $\psi_{C}(x)$ and $\zeta_{C}(s)$, respectively. 
%For more details we refer to \cite{BrouckeDebruyneRevesz}.
\end{proof}

We end this section by mentioning that there exist Beurling number systems $(\MP, \MN)$ satisfying RH, $\LH(\MP, \Li(x))$, and $\LH(\MN, x)$. Unconditionally, no examples are known satisfying RH and $N(x) = Ax+O(x^{1/2})$. However, in \cite[Theorem 3.1]{BrouckeVindas} the existence of systems satisfying
%\footnote{Actually, even the existence of systems with $\psi(x) = x + O((\log x)(\log_{2}x))$ and the same asymptotic behavior for N can be shown.}
\[
	\psi(x) = x + O(x^{1/2}), \quad N(x) = x + O\bigl(x^{1/2}\exp(c(\log x)^{2/3})\bigr), \quad \text{for some } c > 0
\]
was shown. This is done by applying the probabilistic procedure \cite[Theorem 1.2]{BrouckeVindas} or Theorem \ref{probabilistic theorem} of this paper with $M(x) = \Li(x)$ to generate a sequence of Beurling primes $\MP$ satisfying \eqref{eq: probabilistic bound}. By construction, RH and $\LH(\MP, \Li(x))$ are satisfied. If $\MN$ is the sequence of Beurling integers generated by $\MP$ with counting function $N(x)$, we readily have $N(x)\sim Ax$ for some $A>0$. In \cite{BrouckeVindas} it was shown that the bound \eqref{eq: probabilistic bound} implies that the associated Beurling zeta function $\zeta(s)$ has analytic continuation to $\Re s>1/2$, apart from a simple pole at $s=1$, and satisfies the following bound:
\begin{equation}
\label{eq: zeta bound}
	\zeta(\sigma+\I \tau) \ll \exp\biggl\{C\biggl(\frac{1}{\sigma-1/2}+\sqrt{\frac{\log\abs{\tau}}{\sigma-1/2}}\biggr)\biggr\}, \quad \text{for some } C>0 \text{ and if } \sigma >\frac12, \quad \abs{\tau}>1.
\end{equation}
By changing finitely many primes, one can assume that the residue of $\zeta(s)$ at $s=1$ equals $1$, so that $N(x)\sim x$. The error term $O\bigl(x^{1/2}\exp(c(\log x)^{2/3})\bigr)$ then follows from the bound \eqref{eq: zeta bound} and Perron inversion. (See the proof of \cite[Theorem 3.1]{BrouckeVindas} for more details.) 

Actually $\LH(\MN, x)$ can be proven in a similar way. 
Let $\eps>0$ and $B\ge1$ be arbitrary. From \eqref{eq: zeta bound} it follows that there is a $\lambda = \lambda(\eps) >0$ such that 
\[
	\zeta\Bigl(\frac{1}{2} + \frac{\lambda}{\log T} + \I \tau\Bigr) \ll T^{\eps}, \quad \text{for } \abs{\tau}\le 2T.
\]
Let $x$ and $t$ be given with $1\le x\le \abs{t}^{B}$. Applying the Perron formula \eqref{effective Perron} with $\kappa=1+1/\log x$, $T=\abs{t}^{B}$, and shifting the contour of integration to the line $\Re s = 1/2 + \lambda/\log T$, one sees after a small calculation that
\[
	\sum_{n_{j}\le x }n_{j}^{-\I t} = \frac{x^{1-\I t}}{1-\I t} + O_{\eps}\bigl(x^{1/2}\abs{t}^{2\eps B}+x^{1/2}\exp(c(\log x)^{2/3})\bigr).
\]
This establishes $\LH(\MN, x)$.
%%%%%%%%%%%%%%%%%%%%%%%%%%%%%%%%%%%%%%%%%%%%%%%%%%%%%%%%%%%%%%%%%%%%%%%%%%%%%

\section{Remarks on the definition of $\LH(\MN, M(x))$}
\label{sec: remarks definition}
As touched upon in the Introduction, the definition of $\LH$ is not intrinsic to the considered sequence, but also depends on the function $M(x)$. For example, if $\MN$ is a non-decreasing sequence with $N(x) = x + O_{A}(x/\log^{A} x)$ for every $A>0$, then it is admissible both with $M_{1}(x) = x$ and $M_{2}(x) = x + x^{3/4}$. However, $\LH(\MN, M_{1}(x))$ and $\LH(\MN, M_{2}(x))$ are mutually exclusive. Moreover, if $\MN$ does not have too large clusters, one can always artificially construct $M(x)$ such that $\LH(\MN, M(x))$ holds.

\begin{proposition}
\label{prop: choice M}
Let $\MN =(n_{j})_{j\ge1}$ be a non-decreasing sequence with $N(x) \asymp x$, and suppose that $N(x+\e^{-cx}) - N(x-\e^{cx}) \ll \sqrt{x}$ for some $c>0$. Then there is a non-decreasing smooth function $M$ such that $N(x) = M(x) + O(\sqrt{x})$ and such that for every $\eps>0$ we have that for $x$, $\abs{t}\ge 1$,
\[
	\sum_{n_{j}\le x}n_{j}^{-\I t} = \int_{n_{1}}^{x}u^{-\I t}M'(u)\dif u + O_{\eps}(x^{1/2}\abs{t}^{\eps}).
\]
\end{proposition}
\begin{proof}
The idea is to define $M$ such that $M'$ approximates $\sum_{j}\delta_{n_{j}}(x)$, where $\delta_{y}(x)$ denotes the Dirac delta supported at $y$. Let $\vphi$ be  a smooth non-negative function with $\supp \vphi \subseteq [-1,1]$ and $\int\vphi=1$. For $\lambda>0$ we set $\vphi_{\lambda}(x) = \lambda\vphi(\lambda x)$ and
\begin{align*}
	\psi(x) 	&= \sum_{j\ge1}\vphi_{\lambda_{j}}(x-n_{j}), \quad \lambda_{j} = \e^{Cn_{j}}; \\
	M(x)		&=\int_{-\infty}^{x}\psi(u)\dif u;
\end{align*}
where $C\ge\max\{1,2c\}$. By the ``clustering condition'' $N(x+\e^{-cx}) - N(x-\e^{-cx}) \ll \sqrt{x}$, one sees that $N(x)=M(x)+O(\sqrt{x})$. Let now $\eps>0$ and $x, t \ge 1$. If $t \ge \e^{\sqrt{x}}$, then $t^{\eps} \ge \e^{\eps\sqrt{x}} \gg_{\eps} x$, so that the exponential sum $\sum_{n_{j}\le x}n_{j}^{-\I t}$  and the integral $\int_{n_{1}}^{x}u^{-\I t}M'(u)\dif u$ are both trivially $O_{\eps}(x^{1/2}t^{\eps})$.  Suppose now that $t < \e^{\sqrt{x}}$. Then as $N(x) \ll x$ and by the clustering condition we have
\[
	\int_{n_{1}}^{x}u^{-\I t}M'(u) \dif u = \sum_{\sqrt{x}<n_{j}\le x}\int_{n_{j}-1/\lambda_{j}}^{n_{j}+1/\lambda_{j}}u^{-\I t}\vphi_{\lambda_{j}}(u-n_{j})\dif u+ O(\sqrt{x}).
\]
Now for $\abs{u-n_{j}}\le 1/\lambda_{j}$ with $n_{j}\ge\sqrt{x}$, we get
\[
	\abs[1]{u^{-\I t}-n_{j}^{-\I t}} \ll \frac{t}{n_{j}\lambda_{j}} \ll \frac{\e^{\sqrt{x}}}{\sqrt{x}\e^{C\sqrt{x}}} \le \frac{1}{\sqrt{x}}.
\]
We conclude that
\[
	\int_{n_{1}}^{x}u^{-\I t}M'(u) \dif u = \sum_{\sqrt{x}<n_{j}\le x}n_{j}^{-\I t} + O(\sqrt{x}) = \sum_{n_{j}\le x}n_{j}^{-\I t} + O(\sqrt{x}).
\]
\end{proof}

If one would like to modify the definition of $\LH$ to avoid such pathologies, an option is to only allow functions $M(x)$ from a certain class $\mathcal{M}$ of ``sufficiently nice'' functions. Ideally one would also have that $\LH(\MN, M_{1}(x)) \iff \LH(\MN, M_{2}(x))$ for every $M_{1}, M_{2}\in \mathcal{M}$. We were however unable to come up with a satisfactory and unartificial class $\mathcal{M}$.
\medskip

Another option is to propose an alternative definition of $\LH$ which does not refer to a function $M(x)$ at all. For example, if $\MN$ is a sequence satisfying $\LH(\MN, x)$, its exponential sums exhibit square-root cancellation if $\abs{t}\ge x^{1/2}$. So one might define for sequences $\MN$ with counting function $N(x)$:
\[
	\widetilde{\LH}(\MN) \iff \forall \eps>0: x\ge n_{1} \text{ and } \abs{t}\ge N(x)^{1/2} \implies \sum_{n_{j}\le x}n_{j}^{-\I t} \ll_{\eps} N(x)^{1/2}\abs{t}^{\eps}.
\] 

This definition is intrinsic to the sequence. Furthermore, if $M(x)$ is twice differentiable with $\int_{n_{1}}^{x}u|M''(u)|\dif u \ll M(x)$ (in particular $M'(x)\ll M(x)$), then $\LH(\MN, M(x)) \implies \widetilde{\LH}(\MN)$, as can be seen from integrating by parts in $\int_{n_{1}}^{x}u^{-\I t}M'(u)\dif u$.

The results from Section \ref{sec: LH general sequences} hold also for this alternative notion of LH: the example $\MN$ from Theorem \ref{counterexample conjecture 1} does not satisfy $\widetilde{\LH}(\MN)$, in the considered probability spaces $\widetilde{\LH}(\MN)$ holds with probability $1$ (provided that $M$ is sufficiently nice), and in the considered topological spaces $\widetilde{\LH}(\MN)$ only holds on a meagre set (take e.g.\ $F=0$ in Theorem \ref{categoric theorem}).

However, the equivalence $\widetilde{\LH}(\MP) \iff \mathrm{RH}$ does not seem to hold: if $(\MP, \MN)$ is a Beurling number system satisfying\footnote{It is unclear if such systems exist, but one can apply Theorem \ref{probabilistic theorem} with $M(x) = \Li(x)-\Li(x^{3/4})$ to show the existence of a system $(\MP, \MN)$ satisfying $\LH(\MP, \Li(x)-\Li(x^{3/4}))$ and hence also $\widetilde{\LH}(\MP)$, $N(x) = x + O\bigl(x^{1/2}\exp(c(\log x)^{2/3})\bigr)$ for some $c>0$, while RH and $\LH(\MP, \Li(x))$ fail.} 
$N(x) = Ax+O(\sqrt{x})$ and say $\LH(\MP, \Li(x)-\Li(x^{3/4}))$, then $\widetilde{\LH}(\MP)$ holds but not RH, as this would imply $\LH(\MP, \Li(x))$ by Theorem \ref{th: LH(P) <=> RH}.
%%%%%%%%%%%%%%%%%%%%%%%%%%%%%%%%%%%%%%%%%%%%%%%%%%%%%%%%%%%%%%%%%%%%%
\section{Appendix}
\label{sec: Appendix}

%In this appendix we show, that for an admissible sequence $\MN$ with counting function $N(x) = Ax + O(x^{\theta})$, where $\theta < 1/2$, the statement $\LH(\MN, Ax)$ is equivalent to the natural generalization of the classical Lindelöf Hyothesis to $\MN$, namely that $\zeta(1/2 + \I t) \ll_\eps t^\eps$ for all $\eps > 0$. 
Classically, the Lindel\"of hypothesis is formulated as a growth bound for the Riemann zeta function on the line $\Re s=1/2$. Here we discuss the analog for general sequences, which are sufficiently well-behaved so that their Dirichlet series admits analytic continuation to $\Re s \ge 1/2$ (apart from the pole at $s=1$).

Let $\MN = (n_j)_{j\ge1}$ be a non-decreasing sequence with counting function $N(x) = Ax + O(x^{\theta})$, where $0 \leq \theta < 1$. We denote\footnote{Even though the sequence $(n_{j})_{j\ge1}$ is not necessarily the integer sequence generated by a sequence of Beurling primes.} its Dirichlet series as $\zeta(s) = \sum_{j\ge1}n_{j}^{-s}$. Then $\zeta(s) - A/(s-1)$ has analytic continuation to $\Re s >\theta$. We need the following approximation of $\zeta(s)$ in the critical strip, which can be easily verified by adapting the proof of Lemma \ref{lem: convexity bounds}:
\begin{equation}\label{eq: approximation zeta}
	\zeta(s) = \sum_{n_j \leq x} n_j^{-s} - A \frac{x^{1-s}}{1-s} + O\Big(\Big(1+\frac{\abs{s}}{\sigma - \theta}\Big) x^{\theta-\sigma}\Big), \quad \text{for } x \geq 1, \quad \sigma = \Re(s) > \theta.
\end{equation}
%\bl{Check after proof of Lemma \ref{lem: convexity bounds}!}

\begin{theorem}
\label{LH <=> zeta bound}
	Assume $N(x) = Ax + O(x^\theta)$, where $0 \leq \theta < 1/2$. Then $\LH(\MN, Ax)$ is equivalent to
	\begin{equation}\label{eq: generalization classical LH}
		\zeta(1/2 + \I \tau) \ll_\eps |\tau|^\eps \quad \text{for all } \eps > 0.
	\end{equation}
\end{theorem}
\begin{proof}[Sketch of proof.]
%	Suppose first that (\ref{eq: generalization classical LH}). Since only finitely many members of $\MN$ are smaller than $1$, we can assume $n_1 \geq 1$ for accepting an additional error-term $O(1)$ \bl{[Apprarently needed for Perron's formula; we also need to take care about that in Theorem \ref{th: LH(P) <=> RH}]}. For a positive integer $m$ we choose $x \in (m, m+1]$ such that $|x - n_j| \gg x^{-\theta}$ for all $n_j$. Applying the effective Perron inversion formula \eqref{effective Perron} yields	
%	\[
%		\sum_{n_j \leq x} n_j^{- \I t} = \frac{1}{2\pi \I} \int_{c-\I T}^{c+\I T} \zeta(w + \I t) x^w \frac{\dif w}{w} + O\Big( \frac{x^{2 + \theta} \log x}{T}\Big),
%	\]
%	where $T \geq 1$ and $c = 1 + 1/\log x$. We shift the path of integration left from $c$ to $1/2$, obtaining the dominant term from the pole of $\zeta(w + \I t)$ at $1 - \I t$. The assumtion (\ref{eq: generalization classical LH}) implies $\zeta(s) \ll_\eps |t|^\eps$ for all $s$ with $|s-1|>1$ in the strip $\Re(s) \in [1/2, 1]$ by the Phragmen-Lindelöf principle. We use this the bound the remaining horizontal and vertical integrals. 

Suppose first that (\ref{eq: generalization classical LH}) holds. Then the Phragm\'en--Lindelöf principle implies that
\begin{equation}\label{eq: bound zeta classical LH}
	\zeta(s) - \frac{A}{s-1} \ll_\eps |\tau|^\eps \quad \text{in the vertical strip } \frac{1}{2} \leq \Re(s) \leq 1.
\end{equation}
Now we apply the effective Perron inversion formula \eqref{effective Perron} to the sequence $(a_{j})_{j\ge1} = (n_j^{-\I t})_{j \geq 1}$ and proceed as in the proof of Theorem \ref{th: LH(P) <=> RH}, using (\ref{eq: bound zeta classical LH}) to bound the integrals after shifting the line of integration to the left. This proves $\LH(\MN, Ax)$.
	
Now suppose that $\LH(\MN, Ax)$ holds. Then we obtain from (\ref{eq: approximation zeta}) with $x = |\tau|^B$, where $B = (1/2 - \theta)^{-1}$,
\begin{equation}\label{eq: approximation zeta 1/2}
	\zeta(1/2+\I \tau) = \sum_{n_j \leq x} n_j^{-1/2-\I \tau} - A \frac{x^{1/2-\I \tau}}{1/2 - \I \tau} + O(1).
\end{equation}
By assumption we have 
\[
	S(u) = \sum_{n_j \leq u} n_j^{-\I \tau} = A \frac{u^{1-\I \tau}}{1 - \I \tau} + O_\eps(u^{1/2} |\tau|^\eps)
\]
for $1 \leq u \leq |\tau|^B = x$. Integrating by parts yields	
\[
	\sum_{n_j \leq x} n_j^{-1/2-\I \tau} = \int_{1}^x u^{-1/2} \dif S(u) = A \frac{x^{1/2-\I \tau}}{1/2 - \I \tau} + O_\eps(|\tau|^\eps).
\]
Substituting this in (\ref{eq: approximation zeta 1/2}) gives (\ref{eq: generalization classical LH}) and the proof is finished.
\end{proof}
Note that the proof shows in particular that
\begin{equation}\label{eq: LH for fixed B}
	\sum_{n_{j}\le x}n_{j}^{-\I t} = A\frac{x^{1-\I t}}{1-\I t} + O_{\eps}(x^{1/2}\abs{t}^{\eps}) \quad \text{for } 1 \leq x \leq |t|^B
\end{equation}
with fixed $B = (1/2-\theta)^{-1}$ implies (\ref{eq: generalization classical LH}), which in turn implies (\ref{eq: LH for fixed B}) for all $B>0$. The fact that \eqref{eq: LH for fixed B} with $B = (1/2-\theta)^{-1}$ implies \eqref{eq: LH for fixed B} for any fixed $B>0$ also follows directly by estimating 
\[
	\sum_{\abs{t}^{\frac{1}{1/2-\theta}} < n_{j}\le x}n_{j}^{-\I t}
\]
by integrating by parts.

%\bl{
%Maybe its also worth adding results of this kind:
%Let $\MN= (n_{j})_{j}$ be a sequence with $N(x) = Ax+ O(x^{\theta})$, $A>0$, $\theta\in[0,1/2)$. Let $\zeta_{\MN}(s) = \sum_{n_{j}}n_{j}^{-s}$. It has analytic continuation, apart from a pole at $s=1$, to $\Re s > \theta$. 
%Suppose for every $\eps>0$ and for every $x,t$ with $1\le x\le \abs{t}^{\frac{1}{1/2-\theta}}$ we have 
%\begin{equation}
%\label{LH B=1/(1/2-theta)}
%	\sum_{n_{j}\le x}n_{j}^{-\I t} = A\frac{x^{1-\I t}}{1-\I t} + O_{\eps}(x^{1/2}\abs{t}^{\eps}).
%\end{equation}
%Then 
%\[
%	\zeta_{\MN}(1/2+\I t) \ll_{\eps}\abs{t}^{\eps}, \quad \forall \eps>0.
%\]
%Conversely, if the above holds, then $\LH(\MN, Ax)$ holds. Actually, \eqref{LH B=1/(1/2-theta)} with $1\le x\le \abs{t}^{\frac{1}{1/2-\theta}}$ implies \eqref{LH B=1/(1/2-theta)} for any $1\le x\le \abs{t}^{B}$, as can be seen by integration by parts.}


\begin{thebibliography}{99}

\bibitem{Banks} W.~D.~Banks, \emph{The Riemann and Lindel\"of hypothesis are determined by thin sets of primes}, Proc. Amer. Math. Soc. \textbf{150} (2022), no. 10, 4213--4222.

\bibitem{Beurling} A. Beurling, Analyse de la loi asymptotique de la distribution des nombres premiers g\'{e}n\'{e}ralis\'{e}s, I, {\em Acta Math.} {\bf 68} (1937), 255-291.

%\bibitem{BDV2020} F.~Broucke, G.~Debruyne, J.~Vindas, \emph{Beurling integers with RH and large oscillation}, Adv. Math. \textbf{370} (2020), article 107240.

%\bibitem{BDV2022} F.~Broucke, G.~Debruyne, J.~Vindas, \emph{The optimal Malliavin-type remainder for Beurling generalized integers}, to appear in J. Inst. Math. Jussieu.

%\bibitem{BrouckeDebruyneRevesz} F.~Broucke, G.~Debruyne, Sz.~Gy.~R\'ev\'esz, \emph{Some examples of well-behaved Beurling number systems}, manuscript in preparation.

\bibitem{BrouckeVindas} F.~Broucke, J.~Vindas, \emph{A new generalized prime random approximation procedure and some of its applications}, preprint available on arXiv: 2102.08478.

\bibitem{DiamondMontgomeryVorhauer} H.~G.~Diamond, H.~L.~Montgomery, U.~M.~A.~Vorhauer, \emph{Beurling primes with large oscillation}, Math. Ann. \textbf{334} (2006) 1--36.

\bibitem{DiamondZhangbook} H.~G.~Diamond, W.-B.~Zhang, \emph{Beurling generalized numbers,} Mathematical Surveys and Monographs series,  American Mathematical Society, Providence, RI, 2016. 

\bibitem{GGL} S.~M.~Gonek, S.~W.~Graham, Y.~Lee, \emph{The Lindel\"of hypothesis for primes is equivalent to the Riemann hypothesis}, Proc. Amer. Math. Soc. \textbf{148} (2020), no. 7, 2863--2875.

\bibitem{Hoeffding} W.~Hoeffding, \emph{Probability inequalities for sums of bounded random variables}, J. Amer. Statist. Assoc. \textbf{58} (1963), 13--30.

%\bibitem{Hilberdink2005} T.~W.~Hilberdink, \emph{Well-behaved Beurling primes and integers}, J. Number Theory \textbf{112} (2005), 332--344.

%\bibitem{Knopfmacher} J.~Knopfmacher, \emph{Abstract analytic number theory}, North Holland \& Elsevier, Amsterdam--Oxford \& New York, 1975.

\bibitem{Pintz1980} J.~Pintz, \emph{On the remainder term of the prime number formula I. On a problem of Littlewood}, Acta Arith. \textbf{36} (1980), 341--365.

\bibitem{Pintz1982} J.~Pintz, \emph{Oscillatory properties of $M(x)=\sum_{n\le x}\mu(n)$, I}, Acta Arith. \textbf{42} (1982), 49--55.

\bibitem{Oxtoby} J.~C.~Oxtoby, \emph{Measure and Category}, Second edition, Graduate Texts in Mathematics, 2, Springer-Verlag New York, NY, 1980.

\bibitem{Revesz2021} Sz.~Gy.~R\'{e}v\'{e}sz, \emph{A Riemann--von Mangoldt-type formula for the distribution of Beurling primes}, Math.\ Pann.\ New Series \textbf{27} no.\ 2 (2021), 204--232.

\bibitem{Tenenbaum} G.~Tenenbaum, \emph{Introduction to analytic and probabilistic number theory}, Third edition, Graduate Studies in Mathematics, 163, American Mathematical Society, Providence, RI, 2015.

\bibitem{Zhang} W.-B.~Zhang, \emph{Beurling primes with RH and Beurling primes with large oscillation}, Math. Ann. \textbf{337} (2007), 671--704.
\end{thebibliography}
\end{document}